\documentclass[11pt]{article}
\usepackage{amsmath, amssymb, theorem, latexsym, epsfig}
\numberwithin{equation}{section}

\theoremstyle{plain}
\theorembodyfont{\itshape}
\newtheorem{theorem}{Theorem}[section]
\newtheorem{proposition}[theorem]{Proposition}
\newtheorem{lemma}[theorem]{Lemma}
\newtheorem{corollary}[theorem]{Corollary}
                                                                                
\theorembodyfont{\rmfamily}
\newtheorem{definition}[theorem]{Definition}

\newtheorem{example}[theorem]{Example}
\newtheorem{remark}[theorem]{Remark}
\newtheorem{convention}[theorem]{Convention}
\newenvironment{proof}{{\noindent \textbf{Proof}\,\,}}{\hspace*{\fill}$\Box$\medskip}

\title{On quadrilateral  orbits in complex algebraic planar billiards}
\author{Alexey Glutsyuk
\thanks{Permanent address: CNRS, Unit\'e de Math\'ematiques
Pures et Appliqu\'ees, M.R., \'Ecole Normale Sup\'erieure de Lyon,
46 all\'ee d'Italie, 69364 Lyon 07, France.  \newline Email:
aglutsyu@ens-lyon.fr}
\thanks{Laboratoire
J.-V.Poncelet (UMI 2615 du CNRS and the Independent University of Moscow)}
\thanks{National Research University Higher School of Economics, Russia}
 \thanks{Supported by part by RFBR grants 
10-01-00739-a, 13-01-00969-a
and NTsNIL\_a (RFBR-CNRS)  10-01-93115 and by ANR grant ANR-08-JCJC-0130-01.}}
\begin{document}
\maketitle
{\it To my dear teacher Yu.S.Ilyashenko on the occasion of his 70-th birthday}
\begin{abstract} The famous conjecture of V.Ya.Ivrii (1978) says that 
{\it in every billiard with infinitely-smooth boundary in a Euclidean space 
the set of periodic orbits has measure zero}. In the present paper we study the 
complex algebraic version of Ivrii's conjecture for quadrilateral  
orbits in two dimensions, with reflections from complex algebraic curves. 
We present the complete classification of 4-reflective algebraic 
counterexamples:  billiards formed by four complex algebraic curves in the 
projective plane that have open set of quadrilateral orbits. As a corollary, we provide classification of 
the so-called  real algebraic pseudo-billiards with open set of quadrilateral orbits: 
billiards formed by four real algebraic curves; the reflections allow 
to change the side with respect to  the reflecting tangent line. 
 \end{abstract}
\tableofcontents

\def\cc{\mathbb C}
\def\oc{\overline{\cc}}
\def\oci{\oc_{\infty}}
\def\cp{\mathbb{CP}}
\def\wt#1{\widetilde#1}
\def\rr{\mathbb R}
\def\var{\varepsilon}
\def\tt{\mathcal T}
\def\mcr{\mathcal R}
\def\a{\alpha}
\def\ha{\hat a}
\def\hb{\hat b}
\def\hc{\hat c}
\def\hd{\hat d}
 \def\mct{\mathcal T}
 \def\Int{\operatorname{Int}}
 
\section{Introduction}

\subsection{Main result: the classification of 4-reflective  complex planar algebraic billiards}

The famous V.Ya.Ivrii's conjecture \cite{Ivrii} 
 says that {\it in every billiard with infinitely-smooth 
boundary in a Euclidean space of any dimension the set of periodic orbits 
has measure zero.} As it was shown by V.Ya.Ivrii \cite{Ivrii}, his conjecture implies the famous 
H.Weyl's conjecture  on the two-term 
asymptotics of the spectrum of Laplacian \cite{HWeyl11}. A brief historical survey of 
both conjectures with references is presented in \cite{gk1,gk2}. 

For the proof of Ivrii's conjecture it suffices to show that for every $k\in\mathbb N$ 
the set ot $k$-periodic orbits has measure zero. For $k=3$ this 
was proved in  \cite{bzh, rychlik, stojanov, W}  for dimension two and in \cite{vorobets} for any dimension. 
For $k=4$ in dimension two this was proved in 
\cite{gk1, gk2}. 

\begin{remark}
Ivrii's conjecture is open 
already for piecewise-analytic billiards, and we believe that this is its principal case. 
In the latter case Ivrii's conjecture is equivalent to the statement saying that for every 
$k\in\mathbb N$ the set of $k$-periodic orbits has empty interior.
\end{remark}

 In the present paper we study a complexified version of Ivrii's  
conjecture in complex dimension two. More precisely, we consider the complex plane 
$\cc^2$ with the complexified Euclidean metric, which is the standard complex-bilinear 
quadratic form $dz_1^2+dz_2^2$. This defines notion of symmetry with respect to a complex line, 
reflections with respect to complex lines and more 
generally, reflections of complex lines with respect to complex analytic (algebraic) curves. 
The symmetry is defined by the same formula, as in the real case. More details concerning the complex reflection law 
are given  in  Subsection 2.1. One could have replaced the initial real Euclidean metric by
a pseudo-Euclidean one: the geometry of the latter is  
somewhat similar to that of our complex Euclidean metric. Billiards in pseudo-Euclidean spaces were studied, e.g.,  in \cite{drag, khes}. 
Proofs of the classical Poncelet theorem and its generalizations by using complex methods can be found in  \cite{grifhar, schwa}. 

To formulate the complexified Ivrii's conjecture, let us introduce the following definitions.

\begin{definition} A complex projective line $l\subset\cp^2\supset\cc^2$ is {\it isotropic}, 
if either it  coincides with the infinity line, or the complexified Euclidean quadratic 
form vanishes on $l$. Or equivalently, a  line is isotropic, if it passes through some 
of two points with homogeneous coordinates $(1:\pm i:0)$: the so-called {\it isotropic 
points at infinity} (also known as {\it cyclic} (or {\it circular}) points). 
\end{definition}

\begin{definition} \label{deforb} A {\it complex analytic (algebraic) planar billiard} is a finite collection 
of complex irreducible\footnote{By {\it irreducible} analytic curve we mean an analytic curve holomorphically 
parametrized by a connected Riemann surface.}  analytic (algebraic) 
curves $a_1,\dots,a_k$ that are not isotropic lines; we set $a_{k+1}=a_1$, $a_0=a_k$. 
A {\it $k$-periodic billiard orbit} is a collection of points $A_j\in a_j$,  $A_{k+1}=A_1$, $A_0=A_k$, such that  for every $j=1,\dots,k$ one has 
$A_{j+1}\neq A_j$, the tangent line $T_{A_j}a_j$ is not isotropic 
and the complex lines $A_{j-1}A_j$ and $A_jA_{j+1}$ are 
symmetric with respect to the line $T_{A_j}a_j$ and are distinct from it. (Properly saying, we have to take vertices $A_j$ together with 
prescribed branches of curves $a_j$ at $A_j$: this specifies the  line $T_{A_j}a_j$ in unique way, if $A_j$ is a self-intersection point 
of the curve $a_j$.) 
 \end{definition}

\begin{definition} \label{defref} A complex analytic (algebraic)  billiard $a_1,\dots,a_k$ is {\it $k$-reflective,} if 
it has an open set of $k$-periodic orbits. In more detail, this means that there exists 
an open set of pairs $(A_1,A_2)\in a_1\times a_2$ extendable to $k$-periodic 
orbits $A_1\dots A_k$. (Then the latter property automatically holds for every other 
pair of neighbor mirrors $a_j$, $a_{j+1}$.) 
\end{definition}

{\bf Problem (Complexified  version of Ivrii's conjecture).}  
{\it Classify all the  $k$-reflective complex analytic (algebraic)  billiards.}

\medskip

Contrarily to the real case, where there are no  piecewise $C^4$-smooth 4-reflective planar billiards \cite{gk1,gk2}, there exist 
4-reflective complex algebraic planar billiards.  
In the present paper we classify them\footnote{The 4-reflective complex {\it analytic} planar billiards
 will be classified in the next paper.}   (Theorem \ref{tclass} stated at the end of the subsection). 
Basic families of 4-reflective algebraic planar billiards are given below.  Theorem \ref{tclass}  shows that their    
  straightforward analytic extensions cover all the  4-reflective algebraic planar billiards. 

\begin{remark} An $l$-reflective analytic (algebraic) billiard generates $ml$-reflective analytic (algebraic) 
billiards for all $m\in\mathbb N$. Therefore, $k$-reflective billiards exist for all $k\equiv0(mod 4)$. 
\end{remark}

Now let us pass to the construction of 4-reflective complex billiards. The construction comes from the real domain, and 
we use the following relation between real and complex reflection laws in the real domain.

\begin{remark} \label{rskew} In a real billiard the reflection of a ray from the boundary 
is uniquely defined: the reflection is made at the first point where the ray meets the 
boundary. In the complex case,  the reflection of lines with respect to a 
complex analytic curve is a multivalued mapping 
{\it (correspondence)} of the space of lines in $\cp^2$: we do not have a canonical choice of intersection point of a line with the curve. Moreover, the notion of interior domain 
does not exist in the complex case, since the mirrors have real codimension two. Furthermore, the real reflection law also specifies the {\it side} of reflection. 
Namely,  a triple of points $A,B,C\in\rr^2$, $A\neq B$, $B\neq C$, and a line $L\subset\rr^2$ through $B$ 
satisfy the {\it real reflection law,} if the lines $AB$ and $BC$ are symmetric with respect to $L$, and also the points 
{\it $A$ and $C$ lie in the same half-plane} with respect to the line $L$. The  
{\it complex reflection law} says only that the {\it complex lines} $AB$ and $BC$ are symmetric with respect to $L$ and does not 
specify the positions of the points $A$ and $C$ on these lines: they may be arbitrary. 
A triple of real points $A,B,C\in\rr^2$, $A\neq B$, $B\neq C$ 
and a line $L\subset\rr^2$  through $B$ satisfy the complex 
reflection law, if and only if 

- either they satisfy the usual real reflection law (and then $A$ and $C$ lie on the 
same side from the line $L$),

- or the line $L$ is the bissectrix of the angle $ABC$ (and then $A$ and $C$ lie on different 
sides from the line $L$).

In the latter case we say that the triple $A$, $B$, $C$ and the line $L$ satisfy the {\it skew reflection law}. 
\end{remark}

%
%
%

\begin{example} \label{ex-lines} Consider the following complex billiard with four mirrors $a$, $b$, 
$c$, $d$: $a=c$ is a non-isotropic complex line; $b$ is an arbitrary analytic (algebraic)  curve distinct from $a$;  
$d$ is symmetric to $b$ with respect to the line $a$. This complex billiard 
obviously has an open set of 4-periodic orbits $ABCD$, these orbits are symmetric with 
respect to the line $a$, see Fig.1. 
\end{example}
\begin{figure}[ht]
  \begin{center}
   \epsfig{file=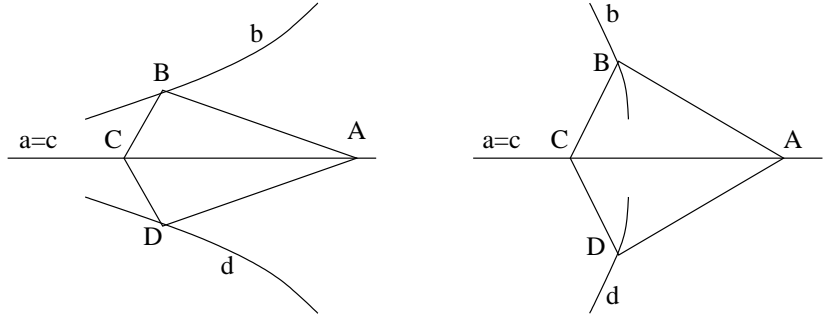}
    \caption{4-reflective billiards symmetric with respect to a line mirror: real pictures}
    \label{fig:1}
  \end{center}
\end{figure}

\begin{example} \label{ex-lines2} Consider a complex 
billiard formed by four distinct lines $a$, $b$, $c$, $d$ passing through the same point $O\in\cp^2$ 
such that {\it the line pairs} $(a,b)$ and $(d,c)$ (or equivalently, $(a,d)$ and $(b,c)$) are unimodularly isometric. That is,  
there exists a complex Euclidean isometry with unit Jacobian that transforms one pair 
into the other: a {\it complex rotation around the intersection point $O$,} see Fig.2. 
This billiard is 4-reflective, or equivalently, the composition of symmetries $\sigma_g$, $g=a,b,c,d$, 
with respect to the lines 
$a$, $b$, $c$ and $d$ that act  on the dual projective plane is identity: 
$\sigma_a\circ\sigma_b\circ\sigma_c\circ\sigma_d=Id$ on $\cp^{2*}$.  Indeed, the latter 
identity is equivalent to the  same identity on the projective plane, i.e., $\sigma_a\circ\sigma_b=\sigma_d\circ\sigma_c$ on $\cp^2$. 
The latter holds if there exists a complex rotation around $O$  sending the pair $(a,b)$ to $(d,c)$. This follows from the fact  that the 
composition of symmetries with respect to two lines through a point $O$ is a complex rotation around $O$ and the commutativity of 
the group of complex rotations around $O$. 
  Let us prove the converse: assuming the 
same identity we show that $(d,c)$ is obtained from $(a,b)$ by complex rotation around $O$. Let $c'$ denote 
the image of the line $b$ under the complex rotation sending $a$ to $d$. One has $\sigma_a\circ\sigma_b=\sigma_d\circ\sigma_{c'}=
\sigma_d\circ\sigma_c$, by the previous statement and assumption. Therefore, $\sigma_{c'}=\sigma_c$, hence $c'=c$. 
\end{example}
\begin{figure}[ht]
  \begin{center}
   \epsfig{file=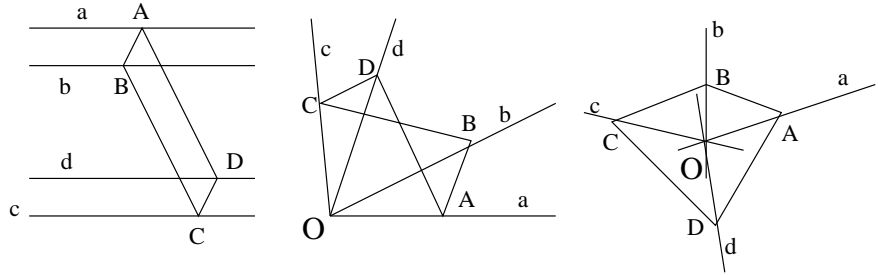}
    \caption{4-reflective billiards on two unimodularly isometric line pairs: real pictures}
    \label{fig:1.5}
  \end{center}
\end{figure}

\begin{example} \label{ex-circles} Consider the complex billiard with four mirrors 
$a$, $b$, $c$, $d$: $a=c$ and $b=d$ are 
complexifications of two distinct concentric circles on the real plane. Say, 
$a=c$ is the smaller circle. We say that we reflect from the bigger circle $b=d$ in the 
usual way, "from interior to interior", and we reflect from the smaller circle 
in the skew way: "from  interior to  exterior" and vice versa.  
Take arbitrary points $A\in a$ and $B\in b$ such that 
the segment $AB$ lies outside the disk bounded by $a$. Let $R$ denote the ray symmetric to the ray 
$BA$ with respect to the diameter through $B$. 
Let $C$ denote its second intersection point with the smaller circle mirror $a=c$. 
(We say that the  ray $R$ passes through the first intersection point "without noticing the  mirror $a$".) 
 Consider the symmetry 
with respect to  the diameter orthogonal to the line $AC$. Let $D\in b=d$ denote 
the symmetric image of the point $B$, see Fig.3. Thus, we have constructed a self-intersected quadrilateral $ABCD$ 
depending analytically on two parameters: $A\in a$ and $B\in b$. 

\begin{figure}[ht]
  \begin{center}
   \epsfig{file=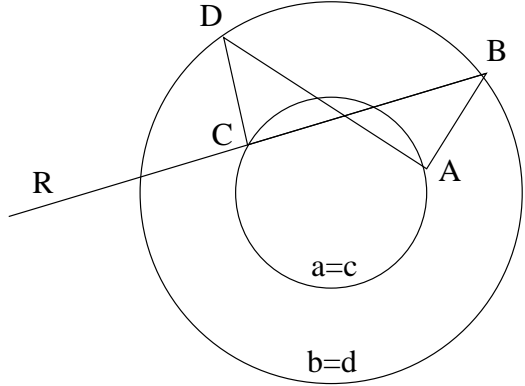}
    \caption{A 4-reflective billiard on concentric circles}
  \end{center}
\end{figure}
\end{example}

{\bf Claim.} {\it The above quadrilateral $ABCD$ is a 4-periodic orbit of the complex billiard 
$a$, $b$, $c$, $d$, and the billiard  is 4-reflective.}

\medskip
\begin{proof} The reflection law is obviously satisfied at the  vertices $B,D\in b=d$. 
It suffices to check the  skew reflection law at the  vertices $A,C\in a$. By symmetry, it suffices to show that 
the tangent line $T_Aa$ is the bissectrix of the angle $BAD$. The union of lines $AD$ and $BC$ intersects the circle $a$ at 
four points with equal intersection angles, since these lines are symmetric with respect to a diameter. 
Similarly, the lines $AB$ and $BC$ 
intersect the circle $a$ at equal angles: they are symmetric with respect to the diameter through $B$. 
The two latter statements together imply 
that the tangent line $T_Aa$ is the bissectrix of the angle $BAD$. Thus, we have a two-parametric quadrilateral orbit family  $ABCD$ that 
extends analytically to complex domain. Hence, the billiard is 4-reflective. This proves the claim.
\end{proof}

Consider a generalization of the above example: a complex billiard $a$, $b$, $c$, $d$ similar to the above one, but now 
$a=c$ and $b=d$ are complexifications of distinct {\it confocal ellipses}, say, $a=c$ is the smaller one. 

\begin{theorem} \label{tell} (M.Urquhart, see \cite[p.59, corollary 4.6]{tab}). 
The above two confocal ellipse  billiard $a$, $b$, $a$, $b$, see Fig.4, is 4-reflective.
\end{theorem}

\begin{figure}[ht]
  \begin{center}
   \epsfig{file=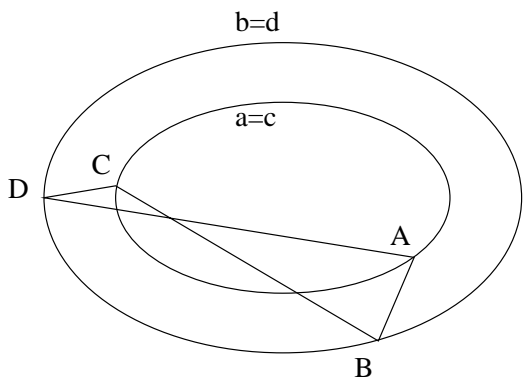}
    \caption{A 4-reflective billiard on confocal ellipses}
    \label{fig:4}
  \end{center}
\end{figure}

The main result of the paper is the following theorem.

%
%
%

\begin{theorem} \label{tclass} A complex planar algebraic billiard $a$, $b$, $c$, $d$ is 
4-reflective, if and only if it has one of the following types: 

Case 1): some mirror, say $a$ is a line, $a=c$, and the curves $b,d\neq a$ are symmetric with respect to 
the line $a$, cf. Example \ref{ex-lines}, Fig.1. 

Case 2):  $a$, $b$, $c$, $d$ are distinct lines through the same point $O\in\cp^2$; the line pair $(a,b)$ is sent to $(d,c)$  by complex rotation 
around $O$, cf. Example \ref{ex-lines2}, Fig.2. 

Case 3):  $a=c$,  $b=d$ and they are distinct confocal conics, 
 cf. Theorems \ref{tell} and \ref{tquad}, Fig.3, 4. 
\end{theorem}

\begin{remark} The notion of confocality of complex conics is the immediate analytic extension to complex domain of 
confocality of real conics. See more precise Definition \ref{defconf} in Subsection 2.4.
\end{remark}


 
%
%
%


\begin{remark} There is an analogue of Ivrii's conjecture in the invisibility theory: Plakhov's invisibility conjecture 
\cite[conjecture 8.2, p.274]{pl}. Its 
complexification coincides with 
the above complexified Ivrii's conjecture \cite{odd}.  For more results on invisibility see \cite{ply, pl, pl2, pl3}. Another analogue of the 
4-reflective planar Ivrii's conjecture is Tabachnikov's commuting billiard problem \cite[p.58, the last paragraph]{tabcom}; 
their complexifications coincide. 
Thus, results about the complexified Ivrii's conjecture have applications not only 
to the original Ivrii's conjecture, but also to Plakhov's invisibility conjecture and Tabachnikov's commuting billiard problem. 
\end{remark}

\subsection{Plan of the proof of Theorem \ref{tclass} and structure of the paper}

Theorem \ref{tclass} is proved in Sections 2--4. In Section 6 we present its application  (Theorem \ref{class}), which gives 
the classification of so-called 4-reflective real algebraic pseudo-billiards: billiards that have open set of 4-periodic orbits with skew reflection 
law at some vertices and usual at the other ones. 

The definition of confocal complex conics and the proof of 4-reflectivity of a billiard $a$, $b$, $a$, $b$ on 
distinct confocal conics $a$ and $b$ (Theorem \ref{tquad}) are given in Subsection 2.4. 
  The 4-reflectivity of  billiards of types 1) and 2) in Theorem 
\ref{tclass} was already explained in Examples \ref{ex-lines} and \ref{ex-lines2}. 

The main part of 
Theorem \ref{tclass} saying that each 4-reflective planar algebraic billiard $a$, $b$, $c$, $d$ 
is of one of the types 1)--3) is proved in Subsection 2.3 and Sections 3, 4. 
The main idea of the proof  is similar to that from  \cite{gk1, gk2}:  to study the ``degenerate'' limits of open set of 
quadrilateral orbits, i.e.,  quadrilaterals having either an edge tangent to a mirror at an adjacent vertex, or 
a pair of coinciding neighbor vertices, or a vertex that is an isotropic tangency point of the corresponding mirror. We deal 
with the compact Riemann surfaces $\ha$, $\hb$, $\hc$, $\hd$, the so-called normalizations of  the curves $a$, $b$, $c$, $d$ respectively 
that parametrize them bijectively except for self-intersections. We study the closure in $\ha\times\hb\times\hc\times\hd$ 
of the open set of 4-periodic orbits in the usual topology. This is a purely two-dimensional algebraic set, which we will call the 
4-reflective set and denote $U$, that contains a Zariski open and dense 
subset $U_0\subset U$ of 4-periodic billiard orbits (Proposition \ref{comp-set}). 
The above-mentioned degenerate quadrilaterals form the complement 
$U\setminus U_0$. The proof of Theorem \ref{tclass} consists of the following steps.

\medskip
Step 1. {\it Description of  a large class of degenerate quadrilaterals \footnote{The complete description of the complement $U\setminus U_0$ 
in a billiard of type 3) will be given by Theorem \ref{irrdeg} in Section 5.} $ABCD\in U\setminus U_0$} (Subsections 2.1 and 2.2). 

1A) Case, when {\it some vertex, say $B$ is an isotropic tangency point} of the corresponding mirror $b$. 
We prove the {\it isotropic reflection law:} if $B\neq A,C$, then 
 at least one of the lines $AB$, $BC$ coincides with the isotropic tangent line $T_Bb$.  
(Propositions \ref{plimits} and \ref{comp-set} in Subsection 2.1) 

1B) Case, when {\it some edge is tangent to the mirror through an adjacent vertex,} say, $AB=T_Bb$. We show that {\it this cannot be the only 
degeneracy} (Proposition \ref{tang}). We show that the case, when $AB=T_Bb$, $b=c$ and $B=C$ is also 
impossible without other degeneracies 
(Proposition \ref{neighbmir}). 
We then deduce (Corollary \ref{lnoncoinc}) that {\it no 4-reflective algebraic planar billiard can have a pair of coinciding neighbor 
mirrors,} say $b=c$. 
This is done by contradiction: 
assuming the contrary, we deform a quadrilateral orbit $ABCD$ to a limit quadrilateral with $B=C$ forbidden by Proposition
\ref{neighbmir}. 

1C) Case, when $AB=T_Bb$ {\it and it is not an isotropic line, and there are no  degeneracies at the neighbor vertices $A$ and $C$.} 
We show that {\it either  $AD=T_Dd$, $A=C$ and $a=c$, or  $D$ is a cusp\footnote{Everywhere in the paper by {\it cusp} we mean 
the singularity of an arbitrary  irreducible singular germ of analytic curve, not necessarily the one given by equation $x^2=y^3+\dots$ 
in appropriate coordinates. The {\it degree} of a cusp is its intersection index with a generic line though the singularity; the 
degree of a regular germ is one.} with non-isotropic tangent line} 
(Corollary \ref{connect} in Subsection 2.2). This follows from  Proposition \ref{tang} and the fact (proved in the same subsection) 
that there exists no one-parameter 
family of quadrilaterals  $ABCD\in U$ with  $C,D\in AB=T_Bb=T_Dd$. Corollary \ref{connect} implies 
Corollary \ref{a=c} saying that {\it if the mirror $b$ is not a line and $d$ has no cusps 
with non-isotropic tangent lines, then $a=c$.} 

\medskip

Step 2. {\it Proof of Theorem \ref{tclass} in the case, when some of the mirrors, say $a$ is a  line (Proposition \ref{straight} 
in Subsection 2.3).} In the subcase, when 
some its neighbor mirror, say $b$ is not  a line, this is done by considering a one-parametric family $\mct$ of degenerate quadrilaterals 
$ABCD\in U$ with $AB=T_Bb$. The above-mentioned Corollary \ref{connect} together with the reflection law at $A$ imply that 
for every $ABCD\in\mct$ one has $AD=T_Dd$, the vertices 
$B$, $D$ are symmetric with respect to 
the line $a$ (an elementary projective duality argument) and $a=c$. Thus,  the billiard is of type 1) in Theorem \ref{tclass}. 
In the subcase, when $a$, $b$ and $d$  are lines, the above arguments applied to $b$ instead of $a$ 
imply that $c$ is also a line. The composition of symmetries with respect to the lines $a$, $b$, $c$ and $d$ acting 
on the dual projective plane is  
identity, by 4-reflectivity. Hence, the billiard is of type 2), if the lines are distinct, and of type 1) otherwise. 

\medskip

The rest of the proof concerns the case, when no mirror is a line.

\medskip

Step 3. {\it Birationality of neighbor edge correspondence and rationality of mirrors} (Subsection 3.1). 
The image of the projection $U\to\hb\times\ha\times\hd$ 
is a two-dimensional projective algebraic variety. It defines an algebraic correspondence $\psi_a:\hb\times\ha\to\ha\times\hd:$ 
$(B,A)\mapsto(A,D)$ for every $ABCD\in U$. First we prove (Lemma \ref{lbirat}) that {\it $\psi_a$ is birational,} by contradiction. 
The contrary assumption implies that some of the transformations $\psi_a^{\pm1}$, say $\psi_a$ has 
two holomorphic branches on some open subset $V\subset \hb\times\ha$:  every $(A,B)\in V$ completes to two distinct 4-periodic orbits 
$ABCD, ABC'D'\in U_0$.  It follows that the quadrilaterals $CDD'C'$ form a two-parameter family of 
 4-periodic orbits of the billiard $c$, $d$, $d$, $c$ with coinciding neighbor mirrors, and the latter billiard is 4-reflective, -- a contradiction 
 to Corollary \ref{lnoncoinc} (Step 1B)). Next we prove that {\it the mirrors are rational curves (Corollary \ref{contrat}).} The proof is based on 
 the observation that for every isotropic tangency point $A\in\ha$ the transformation $\psi_a$ contracts the curve $\hb\times A$ 
 to a point $(A,D)$ with $\pi_d(D)\in T_Aa\cap d$ (follows from the isotropic reflection law, see Proposition \ref{comp-set}, Step 1A)). Applying 
 the classical Indeterminacy Resolution Theorem (\cite{griff}, p.545 of the Russian edition) to the inverse $\psi_a^{-1}$ yields that  
 the curve $b$ is rational. At the end of Subsection 3.1 we deduce Corollary \ref{bircusp}, which deals with one-parametric family 
 $\Gamma$ of quadrilaterals $ABCD\in U\setminus U_0$ with $D$ being a {\it fixed cusp with isotropic tangent line}. It shows that 
 $B\equiv const $ on $\Gamma$ and $B$ is a {\it cusp of the same degree, as $D$.}
 
 \medskip
 
 Step 4. We show that {\it all the mirrors have common isotropic tangent lines and each mirror has the so-called} {\bf property (I):} {\it  
 all the isotropic tangencies are of maximal order, i.e., the intersection of each mirror, say $a$ with its isotropic tangent line corresponds to a single point of its normalization $\ha$} (Lemma \ref{commi} in Subsection 3.2). Namely, given an isotropic tangency point 
 $A\in\ha$ of the curve $a$, 
 we have to show that the curve $d$ intersects the line $L=T_Aa$ at a single point. This is equivalent to the statement saying 
 that for every $D\in\hd$ with $\pi_d(D)\in d\cap L$ the mapping $\psi_a$ contracts the curve $\hb\times A$ to the point $(A,D)$. 
 Or in other words,  the projection to $\hb\times\ha\times\hd$ of the graph of the birational correspondence $\psi_a$ 
 contains the curve $\hb\times A\times D$. The proof of this statement is split into the following substeps. 
 
 4A) Proof of  the {\it local version of the latter inclusion} (Lemma \ref{lemgerm} in Subsection 2.6), 
 which deals with two distinct irreducible 
 germs of analytic curves $(a,A), (b,B)\subset\cp^2$ such that the line $L=T_Aa$ is isotropic and $B\in L$. 
 Its statement concerns the germ of two-dimensional 
 analytic subset $\Pi_{ab}$ in $b\times a\times\cp^{2*}$ at the point $(B,A,L)$ defined as follows: 
 we say that $(B',A',L')\in \Pi_{ab}$, if $A'\in L'$ and either $(A',L')=(A,L)$, or $(B',A')=(B,A)$, or the lines $A'B'$ and $L'$ 
 are symmetric with respect to the tangent line 
 $T_{A'}a$. By definition, the germ $\Pi_{ab}$ contains the germ  of the curve $b\times A\times L$. 
  Lemma \ref{lemgerm} states that  under a mild additional condition {\it the germ $\Pi_{ab}$ is irreducible}. The most technical part of the 
  paper is the proof of Lemma \ref{lemgerm} in the case, when the germs $a$ and $b$ are tangent to each other.  The additional condition 
  from the lemma is imposed only in the case of germs tangent to each other at a finite point  and is formulated in terms of their {\it Puiseaux 
  exponents:} the germs are represented as graphs of multivalued functions in appropriate coordinates, and we take  the lower powers in their Puiseaux expansions. The above additional condition is satisfied automatically, if either the germ $a$ is smooth, or its Puiseaux exponent 
  is no less than that of the germ $b$. The  proof of Lemma \ref{lemgerm} for tangent germs is based on Proposition \ref{tang-asympt} from Subsection 2.6, which deals with the family of tangent lines to $b$.  It relates the asymptotics of the tangency point with that of the 
 intersection points of the tangent line with the curve $a$. 
  
  4B) We cyclically rename the mirrors so that the mirror $a$ has a germ tangent to $L$ that satisfies the condition of Lemma 
  \ref{lemgerm} for every branch of the curve $b$ at its intersection point with $L$. This will be true, e.g., if we name $a$ the curve having a 
  tangent germ 
  to $L$ that is either smooth, or at an infinite point, or with the maximal possible Puiseaux exponent. We show that   {\it each one of the curves  
  $d$ and $b$ intersects the line $L$ at a unique point.} This follows from the irreducibility of the germ $\Pi_{ab}$ (Lemma \ref{lemgerm}),  
  by elementary topological argument given by Proposition \ref{clemgerm}.  
  
  4C) We show that {\it each one of the curves $\ha$ and $\hc$ intersects the line $L$ at a unique point.} 
  In the case, when  the intersection point 
  $D=d\cap L$ is either infinite, or smooth, this follows by applying  Step 4B) to the curve $d$ instead of $a$. Otherwise, if $D$ is 
  a finite cusp, we show that the curves $a$ and $c$  are conics with focus $D$, by using Corollary \ref{bircusp} (Step 3) and 
  Proposition \ref{char-conic} (Subsection 2.4). Finally, we show that $a=c$, by using Corollary \ref{a=c} (Step 1C)) and proving that the 
  curve $d$ has no cusps with non-isotropic tangent lines. The latter statement is proved by applying Riemann-Hurwitz Formula to the 
  projection $d\to\cp^1$ from appropriate point: the intersection point of two isotropic tangent lines to $a$.      
 \medskip
 
 Step 5. We show that {\it the opposite mirrors coincide: $a=c$, $b=d$.} This follows immediately from Corollary \ref{a=c} and absence of 
 cusps with non-isotropic tangent lines in rational curves with property (I). The latter follows from the description of  
 rational curves with property (I) having cusps (Corollary \ref{cnform} in Subsection 2.5). 
 
 \medskip

 Step 6. We prove that {\it the mirrors are conics} (Theorem \ref{caust-conic} in Subsection 4.1). To do this, first we show that 
 they have no cusps (Lemma \ref{lem-cusp}). Assuming the contrary, 
 we show that a mirror with  cusps should have two distinct cusps of equal degrees (basically follows from Corollary \ref{bircusp}, Step 3). 
 This would contradict  the above-mentioned  Corollary \ref{cnform}, which implies that there are  
 at most two cusps and their degrees are distinct. The rest of the proof of Theorem  \ref{caust-conic} is based 
 on the fact that {\it for every $A\in a$ with non-isotropic tangent line $T_Aa$  the collection of the tangent lines to $b$ through $A$ 
 is symmetric with respect to $T_Aa$} (Corollary \ref{cortang}, which follows immediately from Corollary \ref{connect}, Step 2). We 
 apply Corollary \ref{cortang} as $A$ tends to an isotropic tangency point, and deduce from symmetry that the isotropic tangency should be quadratic. 
 In the case, when $a$ and $b$ are tangent to some isotropic line at distinct finite points, this is done by elementary local analysis. 
 In the case, when $a$ and $b$ are isotropically tangent to each other, the proof is slightly more technical and uses 
 Proposition \ref{tang-asympt}, see Step 4A). 
 \medskip
 
 Step 7. We prove that {\it the conics $a=c$ and $b=d$ are confocal} (Subsection 4.2), by using confocality criterion given by 
 Lemma \ref{propconf} in Subsection 2.4.  If one of them is 
 transverse to the infinity line, then their confocality immediately follows from the lemma and the coincidence of their isotropic tangent lines. 
 In the case, when both $a$ and $b$ are tangent to the infinity line, the proof is slightly more technical and is done by using the 
 above-mentioned Corollary \ref{cortang} and Proposition \ref{tang-asympt}.

\section{Preliminaries}

\def\mcl{\mathcal M}

\subsection{Complex reflection law and nearly isotropic reflections.}


We fix an Euclidean metric on $\rr^2$ and consider its complexification: the 
complex-bilinear quadratic form $dz_1^2+dz_2^2$ on the complex affine plane $\cc^2\subset\cp^2$. 
We denote the infinity line in $\cp^2$ by 
 $\oc_{\infty}=\cp^2\setminus\cc^2$.   

\begin{definition} The {\it symmetry} $\cc^2\to\cc^2$ with respect to a non-isotropic 
complex line $L\subset\cp^2$  is the unique non-trivial complex-isometric involution 
fixing the points of the line $L$. 
For every $x\in L$ it acts on the space $\mcl_x=\cp^1$ of lines through $x$, and this action is called {\it symmetry at $x$}. If $L$ is an isotropic line through a finite point $x$, then a pair of  lines through $x$ is called symmetric with respect to $L$, if 
it is a limit of pairs of lines $(l_1^n,l_2^n)$ through points $x_n\to x$ such that $l_1^n$ and $l_2^n$ 
are symmetric with respect to  non-isotropic lines $L_n$ through $x_n$ converging to $L$. 
\end{definition}

\begin{remark} If $L$ is a non-isotropic line, then its symmetry is a projective transformation. Its restriction to 
the infinity line is a conformal involution. The latter is conjugated to the above action on $\mcl_x$ 
via the projective isomorphism $\mcl_x\simeq\oc_{\infty}$ sending a line to its intersection point with $\oc_{\infty}$. 
\end{remark}

\begin{lemma} \label{lim-refl} Let $L$ be an isotropic line through a finite point $x$. Two lines through $x$ are  
symmetric with respect to $L$, if and only if some of them coincides with $L$.  
\end{lemma} 

Let us  introduce an affine coordinate $z$ on $\oc_{\infty}$ in which the isotropic points 
 $I_1=(1:i:0)$, $I_2=(1:-i:0)$ at infinity be respectively $0$ and $\infty$.
As it is shown below, Lemma \ref{lim-refl} is implied by the following proposition

\begin{proposition} \label{tang-asympt1} The symmetry with respect to a finite non-isotropic line  through 
a point $\var\in\oc_{\infty}
\setminus\{0,\infty\}$ acts on $\oc_{\infty}$ by the formula $z\mapsto\frac{\var^2}z$.
 \end{proposition}
 
 \begin{proof} Let $L_\var$ and $\tau_\var$ denote respectively the above line and symmetry. 
 Then $\tau_\var$ acts on $\oc_{\infty}$ by fixing $\var$ and preserving the isotropic point set $\{0,\infty\}$, by definition. 
 It cannot fix $0$ and $\infty$, since otherwise, it would fix three distinct points in $\oc_{\infty}$ and 
hence, would be identity there. Therefore, it would be identity on the whole projective plane, since it also fixes the points of a finite 
line $L_\var$, while it should be a nontrivial involution, -- a contradiction. Thus, $\tau_\var:\oc_{\infty}\to\oc_{\infty}$ 
is a conformal transformation fixing $\var$ and permuting $0$ and $\infty$. Hence, it sends $z$ to $\frac{\var^2}z$. 
This proves the proposition.
\end{proof}
 
\begin{proof} {\bf of Lemma \ref{lim-refl}.} Without loss of generality we consider that $0=I_1\in L$. Let $L^n$ be an arbitrary sequence 
of non-isotropic 
lines converging to $L$, set $\var_n=L^n\cap\oc_{\infty}$: $\var_n\to0$. Let $L_1\neq L$ be another line through $x$, and let 
$L_1^n$ be a sequence of lines converging to $L_1$. Set 
$q_n=L_1^n\cap\oc_{\infty}$, $x_n=L_1^n\cap L^n$: $q_n\to q=L_1\cap\oc_{\infty}\neq0$, $x_n\to x$. 
Then the image $L_2^n$ of the line $L_1^n$ under the symmetry with respect to $L^n$ is the line 
through the points $x_n$ and $p_n=\frac{\var_n^2}{q_n}\in\oc_{\infty}$ (Proposition \ref{tang-asympt1}). 
One has $p_n\to0$. Hence,  $L_2^n\to L$, 
as $n\to\infty$. This proves the lemma. 
\end{proof}

\begin{remark} The statement of Lemma \ref{lim-refl} is wrong in the case, when $L=\oc_{\infty}$. 
Indeed, fix a non-isotropic line $l$ and consider a 
sequence of lines $L^n\to L$ orthogonal to $l$. Then for every $n$ the line symmetric to $l$ with respect to $L^n$ is the line $l$ itself, and 
it does not tend to $L$. On the other hand, the next proposition provides an analogue of Lemma \ref{lim-refl} in the case, 
when the lines $L=\oci$ and $L^n\to L$ are tangent  to one and the same irreducible 
germ of analytic curve. 
\end{remark}

In what follows we deal with irreducible germs of analytic curves that are not lines; we will call them {\it non-linear irreducible germs.} For 
every curve $\Gamma\subset\cp^2$ and $t\in\Gamma$ we identify the tangent line $T_t\Gamma$ with the {\it projective tangent line} to 
$\Gamma$ at $t$ in $\cp^2$.

\begin{definition} An irreducible germ $(\Gamma,O)\subset\cp^2$  of analytic curve has {\it isotropic tangency} 
at $O$, if the projective line tangent to $\Gamma$ at $O$ is  isotropic.  
\end{definition}

\def\mck{\mathcal K}
 \def\az{\operatorname{az}}

\begin{proposition} \label{plimits} Let $(\Gamma,O)$ be a non-linear irreducible germ of analytic curve in $\cp^2$ at its isotropic tangency point 
$O$, set $L=T_O\Gamma\subset\cp^2$. 
Let $l_t$ be an arbitrary family of projective lines through $t\in\Gamma$ that do not accumulate to $L$, as $t\to O$. Let 
$l_t^*$ denote their symmetric images with respect to the tangent lines $T_t\Gamma$. Then $l_t^*\to L$, as $t\to O$. 
%
\end{proposition}

Proposition \ref{plimits} together with its next more precise addendum will be proved below. To state the addendum 
and in what follows, we will use the next 
 convention and definition.

 \begin{convention} \label{convaz} Fix an affine chart in $\cp^2$ (not necessarily the finite plane)  with coordinates $(x,y)$ 
 and the projective coordinate $z=\frac xy$ on its complementary ``new infinity''  line. The {\it complex azimuth} $\az(L)$ of a line 
 $L$ in the affine chart is 
 the $z$-coordinate of its intersection point with the new infinity line. We define the {\it ordered complex angle} $\angle(L_1,L_2)$ between 
 two complex lines as  the difference $\az(L_2)-\az(L_1)$ of their azimuths. We identify the projective  lines 
 $\mathcal M_w\simeq\mathbb{CP}^1$ 
 by translations for all $w$ in the affine chart under consideration. We fix a round sphere metric on $\mathbb{CP}^1\simeq\mathcal M_w$. 
\end{convention}

\begin{definition} \label{defpuis} Consider a non-linear irreducible germ $(\Gamma,O)$ of analytic curve at a point $O\in\cp^2$ and a local chart 
$(x,y)$ centered at $O$ with the tangent projective line $T_O\Gamma\subset\cp^2$ 
being the $x$-axis. Then the curve $\Gamma$ is the graph of a 
(multivalued) analytic function with Puiseaux asymptotics 
\begin{equation} y=\sigma x^r(1+o(1)), \text{ as } x\to0; \  r\in\mathbb Q, \ r>1, \ \sigma\neq0.\label{graph}\end{equation} 
The exponent $r$, which is independent on the choice of coordinates, 
 will be called the {\it Puiseaux exponent} of the germ $(\Gamma,O)$. 
\end{definition}

\medskip

{\bf Addendum to Proposition \ref{plimits}.} {\it In Proposition \ref{plimits} let us measure the angles between lines with respect to an affine 
chart $(x,y)$ centered at $O$ with $L$ being the $x$-axis. Let $r$ be the Puiseaux exponent of the germ $(\Gamma,O)$. Then the azimuth $\az(l_t^*)$ has one of the following asymptotics:

Case 1): the affine chart is the finite plane, thus, $O$ is finite. Then 
\begin{equation}\az(l_t^*)=O(|\frac{y(t)}{x(t)}|^2)=O(|x(t)|^{2(r-1)}), \text{ as }t\to O.\label{az-fin}\end{equation}

Case 2):  the point $O$ is infinite.  Then
\begin{equation} \az(l_t^*)=p\frac{y(t)}{x(t)}(1+o(1))=p\sigma (x(t))^{r-1}(1+o(1)), \text{ as } t\to O; \ p,\sigma\neq0,\label{az-inf}\end{equation}
$$p=\frac r2 \text{ if  } O \neq I_1, I_2; \ p=1 \text{ if } O\in\{ I_1,I_2\} \text{  and } L \text{ is finite;}$$  
$$ p=\frac{r^2}{2r-1} \text{ if } O\in\{ I_1,I_2\} \text{  and } L=\oc_{\infty}.$$}

\medskip

%

%
%
%

\medskip

In the proof of Proposition \ref{plimits} and its addendum and in what follows we will use the following elementary fact.

 \begin{proposition} \label{pxt} Let 
  $(\Gamma,O)$ be a non-linear irreducible germ at the origin of analytic curve in $\cc^2$ that is tangent to the $x$-axis, and let 
  $r$ be its  Puiseaux exponent. For every $t\in \Gamma$ let $P_t$ denote the intersection point 
 of the tangent line $T_t\Gamma$ with the $x$-axis. Then 
 \begin{equation} x(P_t)=\frac{r-1}rx(t)(1+o(1)), \text{ as } t\to O.\label{r-1r}\end{equation}
 \end{proposition}

\def\mcct{\mathcal C_t}

\begin{proof} {\bf of Proposition \ref{plimits} and its addendum.} The statement of  Proposition \ref{plimits} obviously follows from its addendum, 
since $y(t)=o(x(t))$ by assumption. Thus, it suffices to prove the addendum. 


Case 1): the affine chart is finite. Without loss of generality we consider that the isotropic line $L$ passes 
through the point $I_1$. Then formula (\ref{az-fin})  
 follows from Proposition \ref{tang-asympt1} with $\var=\az(T_t\Gamma)\simeq r\frac{y(t)}{x(t)}$.
 
 Case 2): the point $O$ is infinite. Let $P_t$, $Q_{t}^*$, $Q_{t}$ 
denote respectively the points of intersection of the (true) infinity line $\oc_{\infty}$  with the lines 
$T_t\Gamma$, $l_t^*$,  $l_t$. 
 
 Subcase 2a): $O\neq I_{1,2}$. Then $L=\oc_{\infty}$. Recall that $L$ is the $x$-axis, set $u_t=x(t)$. One has 
\begin{equation}x(Q_t)-u_t=o(u_t), \  x(P_t)=qu_t(1+o(1)), \text{ as } t\to O; \ q=\frac{r-1}r\neq1.\label{xptqu}\end{equation}
The former equality follows from the condition of Proposition \ref{plimits}: the line $l_t$ through $t$ and $Q_t$ 
has azimuth bounded away from below, since it does not accumulate to $L$. The latter equality follows from (\ref{r-1r}). 
The $x$-coordinates of the points $Q_t$, $P_t$, $Q_t^*$ form asymptotically an 
arithmetic progression: $x(P_t)-x(Q_t)=(x(Q_t^*)-x(P_t))(1+o(1))$. Indeed, the  points $P_t$ and $Q_t^*$ collide to 
$O\neq I_{1,2}$, as $t\to O$, and the symmetry with respect to $T_t\Gamma$ acts by a conformal involution $S_t:L\to L$ fixing $P_t$,  
permuting $Q_t$ and $Q_t^*$ and permuting distant isotropic points $I_{1,2}$  (reflection law).  The involutions $S_t$ obviously 
converge to the limit conformal involution $S_O$ fixing $O$ and permuting $I_{1,2}$. This implies the above statement on asymptotic 
arithmetic progression. The latter in its turn together 
with (\ref{xptqu}) implies that    $x(Q_t^*)=u_t(s+o(1))$,  $s=2q-1=\frac{r-2}r\neq1$, as $t\to O$. Finally, the line $l_t^*$ passes through 
the points $t=(x(t),y(t))$ and $Q_t^*=(x(t)(s+o(1)),0)$, hence $\az(l_t^*)=p\frac{y(t)}{x(t)}(1+o(1))$, $p=\frac1{1-s}=\frac r2$. 
This proves (\ref{az-inf}).

Subcase 2b): $O$ is an isotropic point at infinity, say $O=I_1$, and the line $L=T_O\Gamma$ is finite. We  work in 
the above coordinate $z$ 
on the infinity line $\oc_{\infty}$, in which $O$ is the origin. We choose coordinates $(x,y)$ so that $\oc_{\infty}$ is the $y$-axis and 
$y=z$ there. Here and in the next paragraph we identify the points $P_t$, $Q_t^*$, $Q_t$ with their 
$z$-cordinates. One has 
\begin{equation}Q_t^*=\frac{P_t^2}{Q_t},\label{pqt}\end{equation}
by Proposition \ref{tang-asympt1}. One has $P_t=O(|x(t)|^r)$, by definition and transversality of the lines $L$ and $\oc_{\infty}$;  
$x(t)=O(Q_t)$, since the line $l_t$ does not accumulate to 
$L=T_O\Gamma$. Hence, $Q_t^*=O(|x(t)|^{2r-1})=o(y(t))$, by (\ref{pqt}). This implies that the line $l_t^*$ through the points 
$t=(x(t),y(t))$ and $Q_t^*=(0,o(y(t)))$ has azimuth of order $\frac{y(t)}{x(t)}(1+o(1))$ and proves (\ref{az-inf}). 

Subcase 2c): $O=I_1$ and $L=\oc_{\infty}$. We choose the local coordinates $(x,y)$ centered at $O$ so that $L$ is the $x$-axis and 
$x=z$ there. One has $Q_t=x(t)(1+o(1))$, $P_t=\frac{r-1}r x(t)(1+o(1))$ as in Subcase 2a). Hence, $Q_t^*=qx(t)(1+o(1))$, 
$q=(\frac{r-1}r)^2$, 
by (\ref{pqt}). Therefore, the line $l_t^*$ through the points $t=(x(t),y(t))$ and $Q_t^*=(qx(t)(1+o(1)),0)$ has azimuth with asymptotics 
$p\frac{y(t)}{x(t)}(1+o(1))$, $p=\frac1{1-q}=\frac{r^2}{2r-1}$. This proves (\ref{az-inf}) and finishes the proof of Proposition 
\ref{plimits} and its addendum.
 \end{proof}

 We will apply Proposition \ref{plimits} to study limits of periodic orbits in complex billiards. To do this, we will use the following convention. 
  
 \begin{convention} \label{conv2} An irreducible analytic (algebraic) curve $a\subset\cp^2$ may have singularities: self-intersections or cusps. 
 We will denote by $\pi_a:\hat a\to a$ its analytic parametrization by an abstract connected Riemann surface  $\hat a$ that is 
 bijective except for self-intersections. It is usually called 
 {\it normalization}. In the case, when $a$ is algebraic, the Riemann surface $\ha$ is compact.  Sometimes 
we idendify a point (subset) in $a$ with its preimage in the normalization $\hat a$ and denote 
both subsets by the same symbol. 
In particular, given a subset in $\cp^2$, 
say a line $l$, we set $\hat a\cap l=\pi_a^{-1}(a\cap l)\subset\hat a$. 
 If $a,b\subset\cp^2$ are two  curves, and $A\in\hat a$, $B\in\hat b$, $\pi_a(A)\neq\pi_b(B)$, then for simplicity we write $A\neq B$ and 
 the line $\pi_a(A)\pi_b(B)$ will be referred to, as $AB$. 
\end{convention}

\begin{definition} \label{defbranch} Let $a$ be an irreducible analytic curve in $\cp^2$, and let $\ha$ be its normalization, 
see the above convention. For every $A\in\ha$ the {\it  local branch $a_A$ of the curve $a$ at 
$A$ } is the irreducible germ of analytic curve given by the germ of 
normalization projection $\pi_a:(\ha,A)\to (a,\pi_a(A))$. 
For every line $l\subset\cp^2$ through $\pi_a(A)$ 
the {\it intersection index at $A\in\ha$} of the curve $\ha$ and the line $l$ is  the intersection index 
of the line $l$ with the local branch $a_A$. The tangent line $T_{\pi_a(A)}a_A$ will be referred to, as $T_Aa$. 
\end{definition}

\begin{definition}
 Let $a_1,\dots,a_k\subset\cp^2$ be an analytic (algebraic) billiard, $P_k\subset\hat a_1\times\dots\times\hat a_k$ be the set of 
 $k$-gons corresponding to its periodic orbits. Consider the closure $\overline{P_k}$ in the usual topology. 
We set 
$$U_0=\Int(P_k), \ U=\overline{U_0}\subset\overline{P_k}.$$
The set $U$  will be called the {\bf $k$-reflective set}.
 \end{definition}
 
 
 \begin{proposition} \label{comp-set}
 The   sets $\overline{P_k}$ and $U$ are analytic (algebraic), and $U$  is the union of the two-dimensional irreducible components of the 
 set $\overline{P_k}$.  The billiard is $k$-reflective, if and only  $U$ is non-empty.   In this case 
for every $j$ the projection $U\to\ha_j\times\ha_{j+1}$ is a submersion on an open dense subset in $U$. 
In the $k$--reflective algebraic case the latter 
projection  is epimorphic and the subset $U_0\subset U$ is Zariski open and dense.  For every  
 $A_1\dots A_k\in\overline{P_k}$ and  every $j$ such that $A_{j\pm1}\neq A_j$ the {\bf complex reflection law} holds: 
 
 - if the tangent line $l_j=T_{A_j}a_j$ is not isotropic, then the lines $A_{j-1}A_j$ and $A_jA_{j+1}$ are symmetric with respect to 
 $l_j$;
 
 - if $l_j$ is isotropic, then either $A_{j-1}A_j$, or $A_jA_{j+1}$ coincides with $l_j$.
\end{proposition}

\begin{proof}
The set $P_k$ is the difference  $\Pi_k\setminus D_k$ of the two following analytic sets.  The set $\Pi_k$ is  locally defined by a system of $k$ 
analytic equations on vertices $A_j$ and $A_{j\pm1}$, $j=1,\dots,k$, saying that either the lines $A_jA_{j-1}$ and $A_jA_{j+1}$ are symmetric 
with respect to the tangent line $T_{A_j}a_j$, or the line $T_{A_j}a_j$ is isotropic, or some of the vertices $A_{j\pm1}$ coincides with $A_j$. 
For example, if a $k$-gon $A_1'\dots A_k'\in\Pi_k$ has finite vertices $A_j'$, $A_{j\pm1}'$ for some $j$, then the corresponding $j$-th equation 
defining $\Pi_k$ in its neighborhood can be written as 
$$(y(A_{j+1})-y(A_j))(y(A_{j-1})-y(A_j))=(\operatorname{az}(T_{A_j}a_j))^2(x(A_{j+1})-x(A_j))(x(A_{j-1})-x(A_j)).$$
The set $D_k$ consists of the $k$-gons having either some of the latter degeneracies (isotropic tangency or neighbor vertex collision), or 
 a vertex $A_j$ such that  $A_jA_{j+1}=T_{A_j}a_j$. 
The set $\overline{P_k}$ is the union of those irreducible components of 
the set $\Pi_k$ that intersect $P_k$. Hence, it is analytic (algebraic). Similarly, the set $U$ is analytic (algebraic), and it is the union of 
two-dimensional irreducible component of the set $\overline{P_k}$. The $k$-reflectivity criterion  
and submersivity  follow from definition; the epimorphicity  in the algebraic case follows from compactness. The (Zariski) openness and 
density of the subset $U_0\subset U$ is obvious. 
The reflection law follows from definition and 
Proposition \ref{plimits}. In more detail, let $A_1^n\dots A_k^n\to A_1\dots A_k$ be a sequence of $k$-periodic orbits converging in 
$\overline{P_k}$. Let for a certain $j$ one have $A_{j\pm1}\neq A_j$ and the tangent line $l_j=T_{A_j}a_j$ be isotropic. 
If $A_{j-1}A_j=l_j$, then we are done. Otherwise $A_jA_{j+1}=l_j$, since the image  $A_j^nA_{j+1}^n$ 
of the line $A_{j-1}^nA_j^n$ under the symmetry with respect to $T_{A_j^n}a_j$ tends to $l_j$, by Proposition \ref{plimits}. 
This proves Proposition \ref{comp-set}.
\end{proof}

\subsection{Tangencies in $k$-reflective billiards}

  Here we deal with (germs of) analytic $k$-reflective planar billiards $a_1,\dots,a_k$ in $\cp^2$: the 
mirrors are (germs of) analytic curves with normalizations $\pi_{a_j}:\hat a_j\to a_j$, see Convention \ref{conv2};  
$\hat a_j$ are neighborhoods of zero in $\cc$, $j=1,\dots,k$. Let $U\subset\hat a_1\times\dots\times\hat a_k$ be the 
$k$-reflective set, see Proposition \ref{comp-set}, $U_0=\Int(P_k)\subset U$. 
The main results of the subsection concern degenerate $k$-gons 
$A_1\dots A_k\subset U\setminus U_0$ such that for a certain $j$  the mirror $a_j$ is not a line, $A_{j\pm1}A_j=T_{A_j}a_j$ and the 
latter line is not isotropic. Propositions \ref{tang} and \ref{neighbmir} show that they cannot have  types as at Fig. 5 and 6a). 
We deduce the following corollaries for $k=4$: Corollary \ref{lnoncoinc} saying that every 4-reflective algebraic billiard has no pair of coinciding neighbor mirrors; 
Corollary  \ref{connect} describing the degeneracy at the vertex opposite to tangency; Corollary \ref{a=c} giving a mild sufficient 
condition for the coincidence of opposite mirrors.
%

%
\begin{definition} A point of  a planar analytic curve is {\it marked}, if it is either a cusp, or an isotropic tangency point.  
Given a parametrized curve $\pi_a:\hat a\to a$ as above, a point $A\in\hat a$ is marked, if it corresponds to a marked point 
of the local branch $a_{A}$, see Definition \ref{defbranch}. 
\end{definition}
\begin{proposition} \label{tang} Let $a_1,\dots,a_k$ and $U$ be as above. 
Then $U$ contains no $k$-gon $A_1\dots A_k$ with the following properties:

- each pair of neighbor vertices correspond to distinct points, and no vertex is a marked point;

- there exists a unique $s\in\{1,\dots,k\}$ such that the line $A_sA_{s+1}$ is tangent to the curve $a_s$ at $A_s$, 
and the latter curve is not a line, see Fig.5. 
\end{proposition}

\begin{remark} A  real version of Proposition \ref{tang} is contained in \cite{gk2} 
(lemma 56, p.315 for $k=4$, and its  generalization (lemma 67, p.322) for higher $k$). 
\end{remark}

\begin{figure}[ht]
  \begin{center}
   \epsfig{file=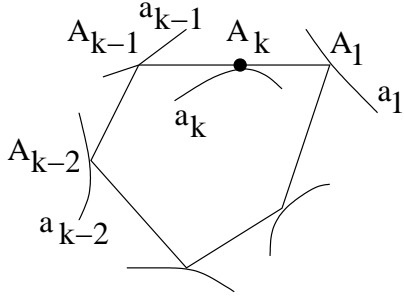}
    \caption{Impossible degeneracy of simple tangency: $s=k$.}
  \end{center}
\end{figure}

\begin{proof} 
Suppose the contrary: there exists a $k$-gon 
$A_1\dots A_k\in U$ as above. Without loss of generality we consider that $s=k$. Moreover, without loss of generality we 
can and will assume that the above tangency is quadratic: the quadrilaterals with a tangency  vertex 
$A_s\in\ha_s$ form a holomorphic curve in $U$ with variable $A_s$. For every $j=1,\dots,k$ the reflection 
with respect to the local branch 
of the curve $a_j$ at $A_j$ induces a  mapping  in the space $\cp^{2*}$ of projective lines. 
More precisely, for every $j\neq k$ it induces a germ of biholomorphic mapping $\psi_j:(\cp^{2*},A_jA_{j-1})\to (\cp^{2*},A_{j+1}A_{j})$, 
since the line $A_jA_{j-1}$ is transverse to $T_{A_j}a_j$ for these $j$. On the other hand, the germ $\psi_k$ is double-valued, with branching locus being the family of lines tangent to $a_k$. Indeed, the image of a line 
close to $A_kA_{k-1}=T_{A_k}a_k$ under the reflection from the curve $a_k$ at their intersection point depends on choice 
of the intersection point. The latter intersection point is a double-valued function with the above branching locus. 
The product  $\psi_k\circ\dots\circ\psi_1$ 
should be identity on an open set accumulating to the line $A_kA_1$, since $A_1\dots A_k$ is a limit of an open set 
of $k$-periodic orbits. But this is impossible, since the product of a biholomorphic germ $\psi_{k-1}\circ\dots\circ\psi_1$ and a 
double-valued germ $\psi_k$ cannot be identity. The proposition is proved.
\end{proof}

 \begin{proposition} \label{neighbmir} Let $a_1,\dots,a_k$ and $U$ be as at the beginning of the subsection. Then 
 $U$ contains no $k$-gon $A_1\dots A_k$ with the following properties:
 
1) each its vertex is not a marked point of the corresponding mirror;

2) there exist $s,r\in\{1,\dots,k\}$, $s<r$ such that 
$a=a_s=a_{s+1}=\dots=a_r$, $A_s=A_{s+1}=\dots=A_r$, and $a$ is not a line;

3)   For every $j\notin \mathcal R=\{ s,\dots,r\}$ one has 
$A_j\neq A_{j\pm1}$ and the line $A_{j-1}A_j$ is not tangent to $a_j$ at $A_j$, see Fig.6a.
\end{proposition}

\begin{figure}[ht]
  \begin{center}
   \epsfig{file=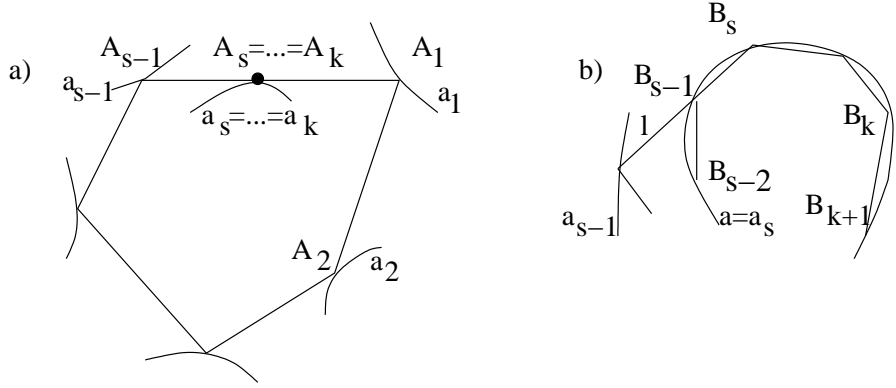}
    \caption{Coincidence of subsequent  vertices and mirrors: $r=k$.}
  \end{center}
\end{figure}

\begin{proof} The proof of Proposition \ref{neighbmir} repeats the above proof with some modifications. 
For simplicity we give the proof  only in the case, when the complement $\{1,\dots,k\}\setminus\mathcal R$ is non-empty. 
In the opposite case the proof is analogous  and is a straightforward complexification of the arguments from \cite{vas}. 
Without loss of generality we consider that $r=k$, then $s\geq2$, and we will assume that the mirror $a=a_s$ has quadratic tangency 
with $T_{A_s}a_s$, as in the above proof.   Consider the germs $\psi_j$ from the above proof. Set 
$$\phi=\psi_{s-1}\dots\psi_1, \ \psi=\psi_k\dots\psi_s.$$
A holomorphic branch of the product $\psi\circ\phi$ following an open set of  $k$-periodic orbits accumulating 
to $\pi_{a_1}(A_1)\dots \pi_{a_k}(A_k)$ should be identity,  as in the above proof. The germ $\phi$ is biholomorphic. 
We claim that the  germ $\psi$ is double-valued on a small neighborhood of the line $A_{s-1}A_s$. 
Indeed, a line $l$ close to $A_{s-1}A_s=T_{A_s}a_s$ and distinct from the latter intersects the local branch $a_{A_s}$ 
at two distinct  points  close to 
 $\pi_{a}(A_s)$. Let us choose one of them and denote it $B_s\in l\cap a_{A_s}$, and denote $B_{s-1}$ the other intersection point. 
 Then the ordered pair $(B_{s-1},B_s)$ extends to a unique orbit $B_{s-1}\dots B_{k+1}$ of length $k+3-s$ 
 of the billiard on the  local branch $a_{A_s}$, see Fig.6b. The line $l^*=B_kB_{k+1}$ is the image of the line $l$ under 
 a branch of the mapping $\psi$. 
 For different choices of the intersection point $B_s$ we get different  lines $l^*$.  Indeed let us fix the above $B_s$ and complete the 
 above orbit to the orbit $B_{2s-k-2}\dots B_{k+1}$ on $a_{A_s}$ of length $2k+4-2s$. Then the output line $l^*$ corresponding to 
 the other intersection point $B_{s-1}$ is the line $B_{2s-k-1}B_{2s-k-2}$, by definition. It is distinct from the line $B_kB_{k+1}$ for an 
 open set of lines $l$, analogously to arguments from  \cite{vas}. This together with the double-valuedness of the intersection point 
 $l\cap a_{A_s}$  proves the double-valuedness of the above germ $\psi$. Each $k$-periodic 
 orbit $B_1\dots B_k$ corresponding to a point in $U_0$ close to $A_1\dots A_k$ 
 contains a sub-orbit $B_s\dots B_k$ on $a_{A_s}$ as above. This implies that the product of the above biholomorphic 
 germ $\phi$ and  double-valued germ $\psi$ should be identity, -- a contradiction. This proves the proposition.
 \end{proof}

\begin{corollary} \label{lnoncoinc} 
 There are no 4-reflective planar algebraic billiards with  a pair of coinciding  
neighbor mirrors.
\end{corollary}

\begin{proof} Suppose the contrary: there exists a 4-reflective planar algebraic 
billiard $a$, $b$, $b$, $d$. Then $b$ cannot be a line, since otherwise, there would exist no 4-periodic orbit $ABCD$:  by 
Definition \ref{deforb},  the lines $AB\neq b$ and $BC=b$ should be symmetric with respect to the non-isotropic line $b$, which is impossible. 
Let $U\subset\hat a\times\hat b\times\hat b\times\hat d$ denote the 4-reflective set. It is two-dimensional and contains at least one 
irreducible algebraic curve $\Gamma$ consisting of quadrilaterals $ABBD$ with coinciding variable vertices  $B=C$, by  epimorphicity of 
the projection $U\to\hb\times\hb$ (Proposition \ref{comp-set}). Let us fix the above $\Gamma$. There are three possible cases:

Case 1): $A\not\equiv B$, $D\not\equiv B$ on $\Gamma$.  Then $AB\equiv BD\equiv T_Bb\not\equiv T_Aa, T_Dd$, since the set of lines 
tangent to two  algebraic curves at distinct points is finite. This implies that $A\not\equiv D$ and a generic quadrilateral $ABBD\in\Gamma$ 
represents a counterexample to Proposition \ref{neighbmir} with $s=2$, $r=3$. 
  
Case 2): $A\equiv B$, $D\not\equiv B$ on $\Gamma$. Then  we get a  contradiction to Proposition \ref{neighbmir} 
with $s=1$, $r=3$. The case $A\not\equiv B$, $D\equiv B$ is symmetric. 

Case 3): $A\equiv B\equiv C\equiv D$ on $\Gamma$. Then we get a contradiction to Proposition \ref{neighbmir} with $s=1$, $r=4$. 
Corollary \ref{lnoncoinc} is proved.
 \end{proof}

 \begin{corollary} \label{connect} 
 Let $a$, $b$, $c$, $d$ be a 4-reflective analytic planar 
 billiard, and let $b$ be not a line. 
 Let $U\subset\hat a\times\hat b\times\hat c\times\hat d$ be the 4-reflective set. Let $ABCD\in U$ be such that 
 $A\neq B$, $B\neq C$, the line $AB=BC$ is tangent to the curve $b$ at $B$ and is not isotropic. Then 
 
 - either $AD=DC$ is tangent to the curve $d$ at $D$,  $\pi_a(A)=\pi_c(C)$, $a=c$ and the  corresponding local branches  
coincide,  i.e., $a_A=c_C$ (see Convention \ref{conv2}): ``opposite tangency connection'', see Fig.7a; 

 - or $\pi_d(D)$ is a cusp of the branch $d_D$ and the tangent line $T_Dd$ is not isotropic: ``tangency--cusp connection'', see Fig.7b. 
  \end{corollary}
 
\begin{figure}[ht]
  \begin{center}
   \epsfig{file=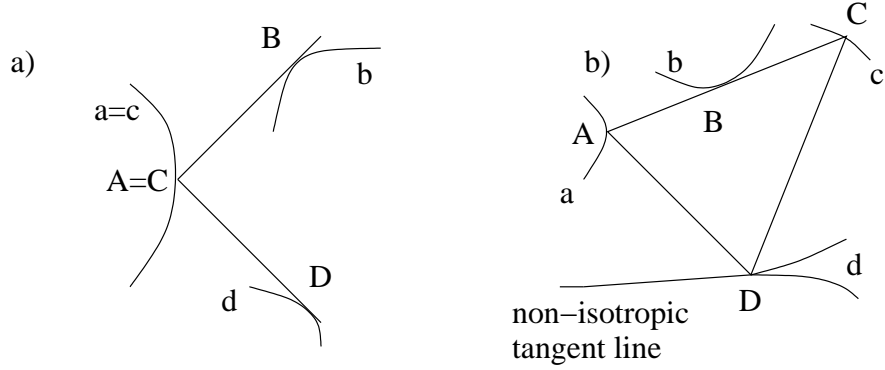}
    \caption{Opposite degeneracy to tangency vertex: tangency or cusp.}
  \end{center}
\end{figure}

%
%
%

 \begin{proof}  Let $\mct\subset U$ denote an irreducible germ at $ABCD$ of analytic curve consisting of quadrilaterals $A'B'C'D'$ such that 
 $A'B'=T_{B'}b$. Then $A'B'\equiv B'C'\equiv T_{B'}b$ on $\mct$, hence 
 $A'$, $B'$, $C'$ vary along the curve $\mct$. Without loss of generality we consider that 
 
 (i)  the line $AB=BC$ has quadratic tangency with the curve 
 $b$ at $B$ and is transverse to $a$ and $c$ at $A$ and $C$ respectively; 
 
 (ii)  the points $A$, $B$, $C$ are finite and not marked;
 
 (iii) $A\neq D$ and $D\neq C$;
 
 (iv) the germ at $ABCD$ of the projection $\pi_{a,b}:U\to\ha\times\hb$ is open. 
 
 One can achieve this by deforming the quadrilateral 
 $ABCD$ along the curve $\mct$. Indeed, it is easy to achieve conditions (i) and (ii). Condition (iv) can be achieved, since 
 the projection $\pi_{a,b}$ is open  at a generic point of the curve $\mct$. Indeed, it is a submersion on an open and dense subset in $U$ 
 (Proposition \ref{comp-set}). Hence, it is open outside at most countable union of  curves contracted to points by $\pi_{a,b}$ (if any), while 
 $\Gamma$ is not contracted.  Condition (iii) can be achieved, since 
$A'\not\equiv D'$ and $D'\not\equiv C'$ on $\mct$. Indeed, if, e.g., $A'\equiv D'$, then the one-parameter family of  lines $A'B'$ would be   
tangent to both curves $a$ and $b$ at distinct points $\pi_a(A')$, $\pi_b(B')$, which is impossible. 
Then either $D'=D$ is a marked point 
 (cusp or isotropic tangency) that is constant on the curve $\mct$, or the line $A'D'=D'C'$  is tangent to $d$ at $D'$ for every 
 $A'B'C'D'\in\mct$. This follows from conditions (i)--(iv), Proposition \ref{tang} and the discreteness of the set of marked points in $\hat d$.  
 
 Case 1): $D'=D$ is a marked point that is constant on $\mct$. Let $D$ be an isotropic tangency point. 
 Then one of the lines $A'D$ or $DC'$ identically coincides with the 
 isotropic tangent line $T_Dd$ (Proposition \ref{comp-set}). 
 Hence, either $A'$, or $C'$ is constant on the curve $\mct$, which is  impossible, see the beginning of the proof. 
 Thus, $\pi_d(D)$ is a cusp and the tangent line $T_Dd$ is not isotropic. 
 
 Case 2): $D$ is not a marked point, the line $AD=DC$ is tangent to $d$ at $D$, and this holds in a neighborhood of the 
 quadrilateral $ABCD$ in the curve $\mct$.  In the case, when $\pi_a(A')\equiv\pi_c(C')$ on $\mct$, one has 
$a_{A}=c_C$, and we are done. Let us show that the opposite case is impossible. 
 Suppose the contrary. Then deforming the quadrilateral $ABCD$ along the curve $\mct$, one can achieve 
 that in addition to  conditions (i)--(iv), one has $A\neq C$. Thus, the line $A'C'$ is identically tangent to both curves $b$ and $d$ 
 at $B'$ and $D'$ respectively along the curve $\mct$. Therefore, $B'\equiv D'$, $b=d$ and the tangent line $A'B'$ to $b$ is orthogonal 
 to both $T_{A'}a$ and $T_{C'}c$ identically on $\mct$ (reflection law). Let us fix a local parameter $t$ on the curve $b$ in a neighborhood of the point $B$. Consider a quadrilateral $A_1B_1C_1D_1\in U_0\subset  U\setminus\mct$ 
 close to $ABCD$. Then $A_1B_1\neq T_{B_1}b$ and $A_1B_1\neq A_1D_1$. Let $L_1$  ($R_1$) denote the line through $A_1$ 
 ($C_1)$ orthogonal 
 to $T_{A_1}a$ (respectively, $T_{C_1}c$). These lines are tangent to $b$ at some points $B_{L,1}$ and $B_{R,1}$ respectively. 
 Set  $\var=t(B_1)-t(B_{L,1})$. 
 Without loss of generality we consider that the point $\pi_d(D)$ is finite, since $D'\equiv B'$ varies along the curve $\mct$. 
 We measure angles between lines in the finite affine chart $\cc^2$.  
 We show that the angle between the tangent lines $T_{B_1}b$ and $T_{D_1}b$ is of order $\var$ on one hand, and of order $O(\var^2)$ 
 on the other hand, as $\var\to0$. The contradiction thus obtained will prove the corollary. We identify a point of the curve $b$ 
 (or its normalization $\hb$) with its $t$-coordinate. 
 The angle between the lines $L_1$ and $A_1B_1$ is of order $\var^2$ (quadraticity of tangency). This and analogous statement for $C_1$ 
together with the reflection law at $A_1$, $C_1$ and $D_1$ imply  the following asymptotics, as $\var\to0$: 
 
 a) $\angle(A_1B_1,A_1D_1)=O(\var^2)$, $\angle(C_1B_1,C_1D_1)=O(\var^2)$;
 
 b) $B_1-D_1=c\var(1+o(1))$, $c\neq0$.
 
 Let us prove statement b) in more detail. The lines $A_1B_1$ and $A_1D_1$ are symmetric with respect to the line $L_1$, hence 
 $\angle(L_1,A_1B_1)\simeq-\angle(L_1,A_1D_1)$. They intersect the local 
 branch $b_{B_{L,1}}$ at the points $B_1$ and $D_1$ respectively. This together with the quadraticity implies that  
 $B_1-B_{L,1}\simeq \pm i(D_1-B_{L,1})$. The latter implies statement b) with appropriate  constant $c\neq0$. 
The angle $\angle(T_{B_1}b,T_{D_1}b)$ is an order $\var$ quantity, by statement b) and quadraticity. 
On the other hand, it is the angle between the symmetry lines for the pairs of lines 
$(A_1B_1, B_1C_1)$ and $(A_1D_1,D_1C_1)$ respectively.  This together with  statement a), 
 which says that the latter pairs are ``$O(\var^2)$-close'', 
 implies that the above angle  should be also of order $O(\var^2)$. The contradiction thus obtained 
 proves the corollary.
 \end{proof}
 
 \begin{corollary} \label{a=c} Let in a 4-reflective complex algebraic planar billiard $a$, $b$, $c$, $d$ the mirror $b$ be not a line and 
 $d$ have no cusps with non-isotropic tangent lines. Then $a=c$.
 \end{corollary}
 
 \begin{proof} There exists  an irreducible algebraic curve $\mct\subset U$  consisting of those quadrilaterals $ABCD$ with variable $B$ for which $AB=T_Bb$: every 4-periodic billiard orbit can be deformed  without changing the vertex $B$ to such a quadrilateral. Let us fix 
 an $ABCD\in\mct$ with $B\neq A,C$ and non-isotropic tangent line $T_Bb$. Its vertex $D$ is not a cusp with non-isotropic tangent line, by assumption. Therefore, 
 $a=c$, by Corollary \ref{connect}.
 \end{proof}

 \subsection{Case of a straight mirror}
 
 Here we apply the results of the previous subsection to 
 classify the 4-reflective algebraic billiards with at least one mirror being a line.
 
 \begin{proposition} \label{straight} 
  Let $a$, $b$, $c$, $d$ be a 4-reflective algebraic 
 billiard in $\cp^2$  such that $a$ is a line. If some of the mirrors $b$ or $d$ is not a line, then $a=c$ and the curves $b$, $d$ are 
 symmetric with respect to the line $a$, see Fig.1. 
 If $a$, $b$, $d$ are lines, then $c$ is also a line, the  lines $a$, $b$, $c$, $d$ pass through the 
 same point,  and the line pairs $(a,b)$, $(d,c)$ are sent one into  the other by complex isometry with unit Jacobian, see Fig.2. 
 \end{proposition} 
 
  \begin{proof} We already know that $b\neq a$, $a\neq d$, by Corollary \ref{lnoncoinc}.  Let us first consider the case, when one of the mirrors 
  $b$, $d$, say $b$ is not a line. Let $U$ be the 4-reflective set, 
 and let $\mct\subset U$ be an irreducible algebraic curve as in the proof of Corollary \ref{connect}: it consists 
 of those quadrilaterals $ABCD$ with variable $B$, for which 
 the line $AB$ is tangent to $b$ at $B$. 
  For every $ABCD\in \mct$ such that $B$ is not a marked point and $A,C\neq B,D$ 
  either the point $\pi_d(D)$ is a cusp of the branch $d_D$ (the same for all $ABCD\in \mct$), or 
  the line $AD$ is tangent to $d$ at $D$. This follows from Corollary \ref{connect}. The first, 
  cusp case is impossible, by a projective duality argument. 
  Indeed, if $D$ were constant along the curve $\mct$, then the lines $AB$ with variable $B$ would intersect at one and the 
   same point $B^*$ symmetric to $\pi_d(D)$ with respect to the line $a$ (the reflection law). 
   On the other hand, for every $ABCD\in \mct$ and any other $A'B'C'D'\in \mct$ close to it the 
   intersection point $AB\cap A'B'=T_Bb\cap T_{B'}b$ tends to $\pi_b(B)$, as $A'B'C'D'\to ABCD$. 
Therefore, $\pi_b(B)\equiv B^*$, hence $B\equiv const$, as $ABCD$ 
   ranges in $\mct$, -- a contradiction. Thus, for every $ABCD\in\mct$ the line $AD$ is tangent to $d$ 
   at the point $D$. Finally, the family of tangent lines $AB$ to $b$ is symmetric to the family of 
   tangent lines $AD$ to $d$ with respect to the line $a$. This implies that the curves $b$ and $d$ are also symmetric: 
   the above argument shows that the intersection points $AB\cap A'B'$ and $AD\cap A'D'$ should 
   be symmetric and tend to $\pi_b(B)$ and $\pi_d(D)$ respectively.   One has $\pi_a(A)\equiv \pi_c(C)$ on $\mct$, hence $a=c$ 
   (Corollary \ref{connect}).   The first statement of Proposition    \ref{straight} is proved. Now let us consider the case, when $a$, $b$ and $d$ 
   are lines. Let us prove the second statement of the proposition. If $c$ were not a line, then $a$ would also haven't been a  line, being symmetric  to $c$ with respect to the line $b$ (the first statement of the proposition), -- a contradiction to our assumption. 
   Therefore, $c$ is a line. The composition of reflections from the lines $a$, $b$, $c$, $d$ is identity as a transformation of the space 
   $\cp^{2*}$ of projective lines (4-reflectivity). This together with the last statement of Example \ref{ex-lines2} 
   proves the second statement of the proposition. 
   \end{proof}
%
%

Let us prove that every 4-reflective billiard with at least one straight mirror is of one of the types 1) or 2). If 
each mirror is a line and some of them coincide, then the billiard is of type 1). Indeed, in this case 
the coinciding mirrors are opposite (Corollary \ref{lnoncoinc}), say $a=c$, and $b$, $d$ are symmetric with respect to the line $a$, by 
the isometry of the pairs $(a,b)$ and $(d,c)$. Otherwise, a billiard with a straight mirror is of type either 1), or 2), by 
Proposition \ref{straight}. This proves Theorem \ref{tclass} in the case of straight mirror.

\subsection{Complex confocal conics}
\def\diag{\operatorname{diag}}

Here we  recall the classical notions of confocality and foci for complex conics.  We  extend  Urquhart's Theorem \ref{tell} and the 
characterization of ellipse as a curve with two given foci to complex conics  
 (Theorem \ref{tquad} and Proposition \ref{char-conic} respectively). Afterwards we state and prove Lemma \ref{propconf} characterizing 
 pairs of confocal complex conics in terms of their isotropic tangent lines. Lemma \ref{propconf} and its proof 
 are based on the classical relations between foci and isotropic tangent lines  
(Propositions \ref{prklein} and Corollary \ref{same-lines}). 

\def\k{\mathcal K}
Let $\k_\rr$ ($\k$) denote the space of all the conics in $\mathbb{RP}^2$ ($\cp^2$) 
including degenerate ones: couples of lines. This is a 5-dimensional real (complex) projective space. 
One has complexification inclusion $\k_{\rr}\subset\k$. 
Let $\k'\subset\k$ denote the set of smooth (non-degenerate) complex conics. 
Consider the subset $\Lambda_\rr\subset\k_\rr\times\k_\rr$ of pairs of  confocal ellipses.  
Let 
$$\Lambda\subset \k\times\k$$ denote the minimal  complex 
projective algebraic set containing $\Lambda_\rr$.

\begin{remark} 
The projective algebraic set $\Lambda$ is irreducible and has complex dimension 6.  It is 
 symmetric, as is $\Lambda_\rr$. The subset 
$$\Lambda'=(\Lambda\setminus\diag)\cap\k'\times\k'\subset\Lambda$$
is Zariski open and dense. These statements follow from definition. 
\end{remark} 

\begin{definition} \label{defconf} Two smooth planar complex conics are {\it confocal,} if  their pair is 
contained in $\Lambda$. 
\end{definition} 
 
\begin{theorem} \label{tquad} For every pair of distinct complex confocal 
conics $a$ and $b$  the complex billiard $a$, $b$, $a$, $b$ is 4-reflective.  
\end{theorem}

\begin{proof} 
Consider  
the  fibration $\pi:F\to \k^2$, $F\subset(\cp^2)^4\times\k^2$: the $F$-fiber over a pair $(a,b)\in\k^2$ is the product $a\times b\times a\times b\subset(\cp^2)^4$. Let $\Sigma\subset F$ denote 
the  set of pairs $(ABCD,(a,b))\in F$ such that  $ABCD$ is an interior point of the  4-periodic orbit set of the billiard $a$, $b$, $a$, $b$. Set 
$$\Lambda''=\pi(\Sigma)\cap\Lambda'\subset\Lambda'.$$
For the proof of Theorem \ref{tquad} it suffices to show 
that $\Lambda''=\Lambda'$. The set $\Sigma$ is a difference of two analytic subsets  in $F$, as in the proof of Proposition \ref{comp-set}. The set $\Lambda''$ is constructible, by the latter statement and Remmert's Proper Mapping Theorem (see, \cite[p.46 of the Russian edition]{griff}). 
The set $\Lambda''$ contains a Zariski dense subset $\Lambda_{\rr}\setminus\diag\subset\Lambda'$, 
by Urquhart's Theorem \ref{tell}.  Hence, $\Lambda''$ contains a Zariski open and dense subset in $\Lambda'$. 
Now it suffices to show that the subset $\Lambda''\subset\Lambda'$ 
is closed in the usual topology. That is, fix an arbitrary sequence $(a_n,b_n)\in\Lambda''$ 
converging to a pair of smooth distinct confocal conics $(a,b)$, and let  us show that the 
billiard $a$, $b$, $a$, $b$  is 4-reflective. To do this, fix an arbitrary pair $(A,B)\in a\times b$ that satisfies the following 
genericity conditions: $A\neq B$; 
$A$ and $B$ are finite and not isotropic tangency points of the corresponding conics; 
the pair of lines symmetric to $AB$ with respect to the lines $T_Aa$ and $T_Bb$ intersect the union $a\cup b$ at eight distinct points that are 
not isotropic tangency points. Pairs $(A,B)$ satisfying the 
latter conditions exist and form a Zariski open subset in $a\times b$, since $a\neq b$. 
The pair $(A,B)$ is a limit of pairs 
$(A^n,B^n)\in a^n\times b^n$ extendable to periodic orbits $A^nB^nC^nD^n$ of the billiard $a^n$, $b^n$, $a^n$, $b^n$, since the 
latter pairs form a Zariski open dense subset in $a^n\times b^n$, by Proposition \ref{comp-set}.
After passing to a subsequence, the 
above orbits converge to a 4-periodic orbit $ABCD$ of the billiard $a$, $b$, $a$, $b$, by construction and genericity assumption. 
Thus, each pair $(A,B)$ from a Zariski open subset in $a\times b$ extends to a quadrilateral orbit, and hence, the 
billiard is 4-reflective. Theorem \ref{tquad} is proved.
\end{proof}

\begin{remark} In the above argument the assumption that  $a\neq b$   is 
important. Otherwise, a priori it may happen that while passing to the limit,  
some neighbor vertices $A^n$ and $B^n$ of a  4-periodic 
orbit collide,  and in the limit we get a degenerate 4-periodic orbit 
with coinciding neighbor vertices. 
\end{remark}
%
 
\begin{proposition} \label{prklein} (\cite[p.179]{klein}, \cite[subsection 17.4.3.1, p.334]{berger}, goes back to Laguerre) For every smooth real conic (ellipse, hyperbola, parabola) each 
its focus lies in two transverse isotropic tangent lines to the complexified conic.
\end{proposition}


\begin{corollary} \label{same-lines} \cite[p.179]{klein} Every two complexified confocal real planar conics have the same 
isotropic tangent lines: a pair of transverse isotropic lines through each focus (with multiplicities, see the next remark). 
\end{corollary}


\begin{remark} For every conic $a\subset\mathbb{CP}^2$ and  $C\in
\mathbb{CP}^2\setminus a$ 
there are two distinct tangent lines to $a$ through $C$. But if $C\in a$, then $T_Ca$ is 
the unique tangent line through $C$. We count it twice, since it is the limit of two 
confluenting tangent lines through $C'\notin a$, as $C'\to C$. If $C\in a$ is an 
isotropic point at infinity, then  $T_Ca$ is a {\it double isotropic 
tangent line to $a$.}
\end{remark}

\begin{corollary} Every smooth complex conic has four isotropic tangent lines with 
multiplicities. 
\end{corollary}


\begin{definition} The {\it complex focus} of a smooth complex conic is an  intersection point of some 
its two distinct isotropic tangent lines. 
\end{definition}

\begin{proposition} \label{char-conic}  Let $P,Q\in\cp^2$ be an unordered pair of points that does not coincide with the 
pair of isotropic points at infinity. Let 
$a\subset\cp^2$ be a parametrized analytic curve distinct from an isotropic line such that for  every $A\in a$ the lines 
$AP$ and $AQ$ are symmetric with respect to the line $T_Aa$. Then the curve $a$ is either a conic with foci $P$ and $Q$, 
or a line with $P$ and $Q$ being symmetric with respect to $a$.
\end{proposition}

\begin{proof} None of the points $P$ and $Q$ is an isotropic point at infinity. Indeed, if $P=I_1$, then $Q=I_2$, by symmetry, -- a 
contradiction to the assumption that $\{ P,Q\}\neq\{ I_1,I_2\}$. 
The curve $a$ is a phase curve of the following double-valued singular algebraic line field $\lambda_{P,Q}$ on $\cp^2$: for every $A\in\cp^2$ the lines $AP$ and $AQ$ are symmetric with respect to the line $\lambda_{P,Q}(A)$.  
The singular set of the latter field is the union of the isotropic lines through $P$ and those through $Q$. 
 Each its phase curve is either a conic with foci $P$ and $Q$, or the symmetry line of the pair $(P,Q)$. 
This follows from the same  statement for real $P$ and $Q$ in the real plane, 
which is classical,  and by analyticity in $(P,Q)$ of the line field family $\lambda_{P,Q}$. The proposition is proved.
\end{proof}


\begin{definition} A {\it transverse hyperbola} is a smooth complex conic in $\cp^2$ transverse to the infinity line. 
A {\it generic hyperbola}  is a smooth 
complex conic  that has four distinct isotropic tangent lines. \end{definition} 

\begin{remark} The complexification of a real conic $a$ is a generic hyperbola, if 
and only if $a$ is either an ellipse with distinct foci, or a hyperbola. The complexification of 
a circle is a non-generic transverse hyperbola through both isotropic points $I_1$, $I_2$, 
with two double isotropic tangent lines  at them intersecting  at its center. 
Each generic hyperbola is a transverse one. 
Conversely, a transverse hyperbola  is a generic one, 
if and only if it contains no isotropic points at infinity. 
A conic confocal to a transverse (generic) hyperbola is also a transverse (generic) hyperbola. 
{\it A generic  hyperbola  (a complexified ellipse or 
hyperbola with distinct foci)} has  {\it four distinct finite complex foci} (including the two real ones). 
\end{remark}

\def\mcc{\mathcal C}

\begin{lemma} \label{propconf} Two smooth conics $a$ and $b$ are  confocal, if and only if 
one of the following cases takes place:

1) $a$ and $b$ are transverse hyperbolas with common isotropic 
tangent lines; 

2) ``non-isotropic parabolas'': $a$ and $b$ are tangent to the infinity line at a common non-isotropic point, and 
their finite isotropic tangent lines coincide;

3) ``isotropic parabolas'': 
 the conics $a$ and $b$ are tangent to the infinity line at a common isotropic point, 
have a common finite isotropic tangent line and are obtained one from the other by translation by a vector parallel to the latter 
finite isotropic tangent line.  
\end{lemma}

The first step in the proof of Lemma \ref{propconf}  is the following proposition.

\begin{proposition} \label{isofoci} Let $a$ be a generic hyperbola. A smooth conic $b$ is confocal to $a$, 
 if and only if  $a$ and  $b$ have common isotropic tangent lines. 
\end{proposition}

\begin{proof} Let $\mck''\subset\mck'$ denote the subset of generic hyperbolas. Set
 $$ \wt\Lambda=\Lambda\cap(\mck''\times\mck').$$
Let $\mathcal C\subset\mck''\times\mck'$ denote the subset of pairs of conics having common isotropic tangent lines. These are 
quasiprojective algebraic varieties. We have to show that $\mathcal C=\wt\Lambda$. 
 Indeed, one has $\mathcal C\supset\wt\Lambda$, by Corollary \ref{same-lines} and minimality of the set $\Lambda$. 
For every quadruple $Q$ of distinct isotropic lines, two through each isotropic point 
at infinity, let $C_Q\subset\k'$ denote the space of smooth conics tangent to the collection $Q$. In other terms, 
the conics of the space $C_Q$ are dual to the conics passing through the given four points dual to the lines in $Q$ in the dual projective 
plane. No triple of 
the latter four points lies in the same line. This implies that  the space $C_Q$ is 
conformally-equivalent to punctured projective line. 
 The space $\mathcal C$ is holomorphically fibered over the four-dimensional  space of the above 
quadruples $Q$ with fibers $C_Q\times C_Q$. This implies that 
$\mathcal C$ is a 6-dimensional irreducible quasiprojective variety containing another 6-dimensional irreducible quasiprojective variety 
$\wt\Lambda$. The latter is  a closed subset in the usual topology of the ambient set $\mathcal C$, by definition. Hence, 
both varieties coincide. The proposition is proved.
\end{proof}


\begin{proof} {\bf of Lemma \ref{propconf}.} 
We will call a pair of conics {\it tangentially confocal,} if they satisfy one of the above statements 1)--3). 
First we show that every pair  $(a,b)$ of confocal conics is  tangentially confocal. Then we prove the converse. 
 In the proof of the lemma we use the fact that {\it every pair of confocal conics is a limit of pairs of confocal generic 
hyperbolas}, since the latter pairs form a Zariski open and dense subset in  $\Lambda$. 
 This together with Proposition \ref{isofoci} implies that {\it every two confocal conics have common isotropic tangent lines 
(with multiplicities)}. 

For the proof of Lemma \ref{propconf} we translate the tangential confocality into the dual language. Let $I_1^*,I_2^*\subset\cp^{2*}$ be 
the dual lines to the isotropic points at infinity, set $p=\oc_{\infty}^*=I_1^*\cap I_2^*$. A pair of smooth conics $a$ and $b$ 
are confocal generic hyperbolas, 
if and only if the dual curves $a^*$ and $b^*$ pass through 
the same four distinct points $A_{11}, A_{12}\in I_1^*$, $A_{21},A_{22}\in I_2^*$ (Proposition \ref{isofoci}). 
Now let $(a,b)$ be a limit of pairs $(a_n,b_n)$ 
of confocal generic hyperbolas, let $A_{ij}^n\in I_i^*$ denote the above points of intersection $a_n^*\cap b_n^*$, and let $a$ be not a 
generic hyperbola.  Let us show that one has some of  cases 1)--3). Passing to a subsequence, 
 we consider that the points $A_{ij}^n$ converge to some limits $A_{ij}\in a^*\cap I_i^*$. Then one of the following holds: 

(i) One has $A_{ij}\neq p$ for all $(i,j)$. Then $p\notin a^*,b^*$, hence 
$a$ and $b$ are transverse hyperbolas, and we have case 1).

(ii) One has  $A_{11}=A_{21}=p$, $A_{12}, A_{22}\neq p$, and the conics $a^*$ and $b^*$ are tangent 
to each other at the limit $p$ of colliding intersection points $A_{11}^n,A_{21}^n\in a_n^*\cap b_n^*$. 
This implies that the curves $a$ and $b$ are as in case 2). 

(iii) One has $p=A_{11}=A_{21}=A_{22}\neq A_{12}$ (up to permuting $I_1$ and $I_2$), and the conics $a^*$ and $b^*$ are tangent to each other with triple contact at the limit $p$ of three colliding intersection points $A_{ij}^n\in a_n^*\cap b_n^*$. The corresponding tangent line 
coincides with $I_2^*$, since the points $A_{21}^n,A_{22}^n\in a_n^*\cap I_2^*$ collide. 
Therefore, the conics $a$ and $b$ are tangent to the infinity line at $I_2$ and have triple contact there between them, and 
they have a common finite isotropic tangent line $A_{12}^*$ through $I_1$. Below we show that then we have case 3). 

(iv) All the points $A_{ij}$ coincide with $p$. Then a smooth conic $a^*$ should be tangent at $p$ to both transverse lines $I_1^*$ and 
$I_2^*$, analogously to the above discussion, -- a contradiction. Hence, this case is impossible.

 \begin{proposition}  \label{ptransl} Let two distinct smooth conics $a,b\subset\cp^2$ be tangent to the infinity line $\oc_{\infty}$ at a common 
 point.  Then they  
 have at least  triple contact there, if and only if they are obtained from each other by translation.
    \end{proposition}
  
  \begin{proof} One direction is obvious: if conics $a$ and $b$ tangent to $\oc_{\infty}$ 
  are  translation images of each other, then they have common tangency point with $\oc_{\infty}$ and  at least triple contact there 
  between them. The latter follows from the elementary fact that in this case they may have at most one finite intersection point, which 
  is a solution of a linear equation. 
   Let us prove the converse. Suppose they are tangent to the infinity line at a common point and have at least triple contact there.  
  Let $A\in a$ and $B\in b$ be arbitrary two finite points with parallel tangent lines. The translation by the vector $AB$ sends the 
  curve $a$ to a curve $a_1$ tangent to $b$ at $B$ that has at least triple contact with $b$ at infinity. If  $a_1\neq b$, then their 
  intersection index is at least 5,  -- a contradiction.  Hence, $a_1=b$. This proves the proposition.
  \end{proof}

 Thus, in case (iii) $a$ and $b$ are translation images of each other, by  
 Proposition \ref{ptransl}, and have  a common 
 finite isotropic tangent line (hence, parallel to the translation vector). Therefore, we have case 3).

For every smooth conic $a$ let 
$C_a$ ($CT_a$)  denote respectively the space of smooth conics confocal (tangentially confocal) to $a$. These are quasiprojective varieties. 
We have  shown above that 
$C_a\subset CT_a$, and for the proof of the lemma it suffices to show that $CT_a=C_a$. In the case, when $a$ is a generic hyperbola, this 
follows from Proposition \ref{isofoci}. 
Note that $dim C_a>0$, since this is true for a Zariski open dense subset in $\mathcal K$ of generic hyperbolas $a$   
(see Proposition \ref{isofoci} and its proof) and remains valid while passing to limits. 
Moreover, the subset $C_a\subset \mathcal K'$ is closed by definition. 
Fix a smooth conic $a$ that is not a generic hyperbola.  
Let us show that  $CT_a$ is  a punctured Riemann sphere. This together with the inclusion $C_a\subset CT_a$ and 
closeness  will imply that $CT_a=C_a$ and prove the lemma. 
We will treat  separately each one of cases 1)--3) (or an equivalent dual case (i)--(iii)). 

Case 3) is obvious: the space $CT_a$ of images of the conic $a$ by translations parallel to a given line 
is obviously conformally equivalent to $\cc$. Let us treat case 2)=(ii). In this case the dual curve $a^*$ intersects the union 
$I_1^*\cup I_2^*$ at exactly three distinct points:  $A_{12}\in I_1^*$, $A_{22}\in I_2^*$ and $p=I_1^*\cap I_2^*=\oc_{\infty}^*$. 
The tangent line $l_p=T_pa^*$ is transverse to the lines $I_j^*=pA_{j2}$, $j=1,2$, since $a^*$ is a conic intersecting each line 
$I_j^*$ at two distinct points. 
The conics  tangentially confocal to $a$ are dual to exactly those conics $b^*$  that pass through the  points $A_{12}$, $A_{22}$, $p$ 
 and are tangent to the line $l_p$ at $p$. The latter three points and line being in generic position,  
the space of  conics $b^*$ respecting them as above is a punctured projective line. In case (i) the proof is analogous 
and is omitted to save the space. Lemma \ref{propconf} is proved.
\end{proof}

\begin{corollary} \label{tan-con} Let two confocal conics $a$ and $b$ be tangent to each other. Then each their 
tangency point lies on the infinite line, 
the corresponding tangent line  is isotropic, and one of the following cases holds:

(i) single tangency point of quadratic contact; either the tangency point is isotropic and the tangent line is finite; or 
it is non-isotropic, and the tangent line is infinite;

(ii) two tangency points, which are the two isotropic points at infinity; the tangent lines are finite;

(iii) single tangency point of triple contact: an isotropic point at infinity, the tangent line is infinite. 
\end{corollary}

\begin{proof}
Let $a$ and $b$ be tangent confocal conics. All their common 
tangent lines are isotropic, since this is true for generic hyperbolas and remains valid after passing to limit. 
Case 1) of the lemma corresponds to Cases (i) (first subcase) or (ii) of the corollary. Case 2) of the lemma corresponds to Case (i), second 
subcase. 
Case 3) of the lemma corresponds to Case (iii) of the corollary. These statements follow from the proof of the lemma (the arguments on 
the points $A_{ij}$ of intersection of the dual conics) and the fact that a tangency of two curves corresponds to a tangency of the dual curves. 
For example, in Case 1) (or equivalently, case (i) from the proof of the lemma) a tangency point $O$ of the conics $a$ and $b$ corresponds to 
a common tangency point of the dual conics $a^*$ and $b^*$ with a line $I_j^*$, $j=1,2$. This implies that $O=I_j$. The other cases are 
treated analogously.   
\end{proof}

\subsection{Curves with property (I) of maximal isotropic tangency}

In this subsection we describe the class of special rational curves having property (I) introduced below
(Proposition \ref{pisotr} and Corollary \ref{cnform}). We show in Subsection 3.2 that mirrors of every 4-reflective algebraic 
billiard without lines belong to this class.    These results will be used  in Section 4.

\begin{definition} \label{defi} We say that a  planar projective algebraic curve $a$ that is not a line {\it has property (I)}, 
if every its isotropic tangent line intersects its 
normalization $\ha$ at a single point $A$, see Convention \ref{conv2}; the  intersection 
index of the curve  $\ha$ with $T_Aa$ at $A$ (see Definition \ref{defbranch}) then equals the degree of the curve $a$. 
\end{definition}

\begin{remark} Every conic has property (I). Corollary \ref{cnform} below shows that the 
converse is not true. A curve $a$ that is not a line has property (I), if and only if its dual $a^*$ satisfies the following statement:

($I^*$) For every $j=1,2$ and $t\in a^*\cap I_j^*$ the germ $(a^*,t)$ is irreducible, the line $T_ta^*$ is the only line through $t$ 
tangent to $a^*$, and $t$ is their unique tangency point.  
\end{remark}

\begin{corollary} Each planar projective curve with property (I) has at least two distinct  isotropic tangent lines.
\end{corollary}

\begin{proof} Let $a$ be a property (I) curve. Its isotropic tangent lines are dual to the points of non-empty intersection 
$a^*\cap(I_1^*\cup I_2^*)$. Hence, the contrary to the corollary would imply that this intersection reduces to 
$s=I_1^*\cap I_2^*$. The germ $(a^*,s)$ is irreducible (the above statement ($I^*$)), and hence is not tangent at $s$, say, to 
the line $I_1^*$. This  implies that $a^*$ should intersect $I_1^*$ at some other point. The contradiction thus obtained proves the corollary.
\end{proof}
 
\begin{proposition} \label{pisotr} Let a rational  curve $a$ have  property (I). Then

(i) either $a$ has at least three distinct isotropic tangency points: then  
it  has no cusps, and at least one  its isotropic tangency point is finite;

(ii) or it has exactly two distinct isotropic tangency points; then   at least one of them is an isotropic point at infinity, and $a$ 
has no cusps  except maybe for some of the two latter points.
\end{proposition}

\begin{proof} The curve $a$ has at least two distinct isotropic tangent lines, by the above corollary. 
The  tangency points should be distinct, since the contrary would obviously contradict property (I). 

Case 1): $a$ has at least three distinct isotropic tangency points 
$A$, $B$, $C$. At least one of them is finite. Indeed, otherwise $A,B,C\in\oc_{\infty}$, and some of them, say $A$ is not 
isotropic. Hence, the curve $a$ is tangent to the infinity line at $A$ and intersects it at $B\neq A$, -- a contradiction to property (I). 
Fix arbitrary two isotropic tangency points, say $A$ and $B$. We show  that the curve $a$ has no cusps distinct from them. Applying this 
to the other pairs $(A,C)$ and $(B,C)$ will imply that $a$ has no cusps at all and will prove (i). Let 
 $l_A$ and $l_B$ denote the  projective lines tangent to $a$ at $A$ and $B$ respectively. 
They intersect $a$ only at $A$ (respectively, $B$), 
by property (I). This implies that $l_A\neq l_B$ and $O=l_A\cap l_B\notin a$. Consider the projection $\pi:\ha\to\cp^1$: the composition of the 
parametrization $\pi_a:\ha\to a$ and the projection  from the point $O$. 
Its global degree  and its  local degrees at its critical points corresponding to $A$ and $B$ are 
equal to the degree of the curve $a$ (property (I)). These are the only critical points, since    they have maximal order and 
$a$ is rational. Hence, $a$ has no cusps distinct from $A$ and 
$B$. This together with the above discussion proves (i). 
%

Case 2): $a$ has exactly two isotropic tangency points  $A$ and $B$. Let us prove (ii). 
As is shown above, $a$ has no cusps distinct from 
$A$ and $B$. The dual curve $a^*$ intersects the union $I_1^*\cup I_2^*$ exactly at two distinct points $l_A^*$ and $l_B^*$. 
%
Thus, $a^*$ intersects one of the lines $I_j^*$, say $I_1^*$ at a unique point $t$. Then $a^*$ is tangent to $I_1^*$ at $t$ (irreducibility of the germ $(a^*,t)$, by ($I^*$)). This implies that $I_1\in a$ and proves (ii) and the proposition. 
\end{proof}

\begin{definition} A system of {\it isotropic coordinates} on $\cc^2$ is 
a system of affine coordinates with isotropic axes.
\end{definition}
\begin{corollary} \label{cnform} 
In case (ii) of Proposition \ref{pisotr} one of the following holds (here $d$ is the 
degree of the curve $a$):

- either the curve $a$ is tangent to $\oc_{\infty}$ at an isotropic point at infinity and has another finite isotropic tangency
point; then in appropriate isotropic coordinates the curve $a$ 
is given by the following parametrization: 
\begin{equation}x=t^p, \ y=t^d, \  t\in\oc, \ 0<p<d;\label{form-pd}
\end{equation}

- or the curve $a$ passes through the two isotropic points at infinity and 
 in appropriate isotropic coordinates it is given by the following parametrization: 
\begin{equation}x=t^{-p}, \ y=t^q, \ t\in\oc, \  p,q\geq1, \ p+q=d.\label{form-pq}
\end{equation}
In both formulas $p$ and $d$ are relatively prime. 
In particular, the  curve $a$ is without cusps, if and only if it is a conic and in the 
above formulas $p=1$, $d=2$ and $p=q=1$ respectively. 
\end{corollary}

\begin{proof} Recall that the curve $a^*$ intersects the union $I_1^*\cup I_2^*$ exactly at two distinct points, and one of the 
intersections $a^*\cap I_j^*$ is a single point, see the end of the above 
proof. At each point of intersection $a^*\cap I_j^*$ the germ of the curve $a^*$ is irreducible, see $(I^*)$. 
  Therefore, we have the following possibilities (up to permuting $I_1^*$ and $I_2^*$):

Case 1): $a^*$ passes through $s=I_1^*\cap I_2^*=\oc_{\infty}^*$, $s=a^*\cap I_1^*$, then $a^*$ is tangent to  $I_1^*$ at $s$ (statement 
($I^*$)), and 
$a^*$ intersects $I_2^*$ at a unique point $t$ different from $s$.  Then $a$ is tangent to the infinity line at $I_1$ and has a finite isotropic 
tangent line $t^*$ through $I_2$. The line $T_ta^*$ does not contain $s$, since otherwise, $T_ta^*=ts=I_2^*$ and $I_1^*$ would be 
two distinct tangent lines to $a^*$ through $s$, -- a contradiction to property $(I^*)$. Hence, the dual $O=(T_ta^*)^*$ to the 
line $T_ta^*$ is a finite point, and it lies in $t^*\cap a$ by duality.  The composition of the 
parametrization $\pi_a:\ha\to a$ with the projection $a\to\cp^1$ from the point $I_1$ is a branched covering $\pi:\ha\to\cp^1$. Either it is  
bijective, or it has exactly two critical points $O$ and $I_1$, since  $a$ has neither cusps distinct from them, nor finite tangent lines through $I_1$. Therefore, taking $t^*$ as the $x$-axis, $O$ as the origin and $OI_1$ as the $y$-axis, 
we get (\ref{form-pd}) after appropriate coordinate rescalings. 

Case 2): $a^*$ intersects each line $I_j^*$ at a unique point $t_j$ and is tangent to $I_j^*$ there (by statement ($I^*$)), and $t_j\neq s$. Hence,  
 $a$ passes through both isotropic points  transversely to the infinity line. Taking isotropic coordinates centered at  
the intersection $t_1^*\cap t_2^*$, we get (\ref{form-pq}) after appropriate rescalings. 

The parametrizations in (\ref{form-pd}) and (\ref{form-pq}) can be chosen bijective; then $p$ and $d$ are relatively prime. 
Hence,  if $d\neq2$, the curve $a$ has at least one cusp: either at an isotropic point at infinity, or 
at the origin (the latter may take place only in  case (\ref{form-pd})). The corollary is proved.
\end{proof}

 \subsection{Reflection correspondences: irreducibility and contraction}
 
 
 In what follows, for every irreducible non-linear germ $(a,A)$ of analytic curve in $\mathbb{CP}^2$ (or briefly, {\it irreducible non-linear germ}) 
 its Puiseaux exponent (see Definition \ref{defpuis}) will be denoted by $r_a=r_a(A)$. 
 The main results of this subsection are the following lemma, proposition and corollary. They will be used in Subsection 3.2 
 in the proof of property (I) 
 of  mirrors of a 4-reflective billiard and coincidence of their isotropic tangent lines. Proposition \ref{tang-asympt} stated below will be 
 used in their proofs and also in Section 4, where we show that the mirrors are confocal conics.
 
 \begin{lemma} \label{lemgerm}  Consider a pair of distinct non-linear irreducible germs $(a,A)$ and $(b,B)$, set $L=T_Aa$. Let $L$ 
 be isotropic and $B\in L$. 
 Let $\Pi_{ab}\subset b\times a\times\cp^{2*}$ denote the 
germ at $(B,A,L)$ of two-dimensional analytic subset defined as follows:
 $(\tau,t,l)\in\Pi_{ab}$,  if and only if $t\in l$ and either $(\tau,t)=(B,A)$, or $t=A$ and $l=L$, or the lines $t\tau$, $l$ are 
  symmetric with respect to $T_{t}a$. (Thus, $\Pi_{ab}$ contains the curve $b\times A\times L$.) Let one of the following conditions hold:
 
 (i) $A\neq B$;
 
 (ii) $A=B$, but $L=T_Aa\neq T_Bb$;
 
 (iii) $A=B$, $L=T_Aa=T_Bb$ and $A$ is an infinite point;
 
 (iv) $A=B$ is a finite point, $L=T_Aa=T_Bb$ and 
 \begin{equation}r_b(2-r_a)<r_a.\label{puiseq}\end{equation}
 Then the germ $\Pi_{ab}$ is irreducible. 
  \end{lemma}
 
 \begin{definition} \label{refbir} We say that three irreducible algebraic curves 
 $b,a,d\subset\cp^2$ form a {\it reflection-birational triple,} if they are not lines, 
  $b\neq a$, $a\neq d$ and 
 there exists a birational isomorphism $\psi_a:\hb\times\ha\to\ha\times\hd$ such that for a non-empty 
 Zariski open set of pairs $(B,A)\in\hb\times\ha$ 
  one has $\psi_a(B,A)=(A,D)$ and the lines $AB$ and $AD$ are symmetric with respect to the tangent line $T_Aa$.
  \end{definition}

 \begin{proposition} \label{clemgerm}  Let $b$, $a$, $d$ be a reflection-birational triple, $L$ be an isotropic tangent line to $a$. For every their 
tangency point $A\in\ha$ and every 
$B\in\hb\cap L$ consider the  germ $\Pi_{ab}$  constructed above for the local branches $a_A$ and $b_B$. 
Let there exist a tangency point $A\in\ha$, $T_Aa=L$, such that for every $B\in\hb\cap L$ the corresponding germ $\Pi_{ab}$ is irreducible. 
Then  the line $L$ intersects the curve $\hd$ at a unique point $D$, and the transformation $\psi_a$ contracts the curve $\hb\times A$ 
to the point $(A,D)$. 
  \end{proposition}
  
 \begin{corollary} \label{plemgerm}  Let $b$, $a$, $d$ be a reflection-birational triple, and let $L$ be an isotropic tangent line to $a$. Let 
 there exist a tangency point $A\in\ha$, $T_Aa=L$, such that for every $B\in\hb\cap L$ the local branches $a_A$ and $b_B$ satisfy one of 
 the conditions (i)--(iv) from Lemma \ref{lemgerm}. Then  the line $L$ intersects the curve $\hd$ at a unique point $D$. 
 \end{corollary}

The lemma and the proposition are proved below. The corollary follows immediately from them. 

  \begin{remark} \label{remiv} If in the above condition (iv) inequality (\ref{puiseq}) does not hold, then the germ $\Pi_{ab}$ is not irreducible, and
   some its irreducible component does not contain the curve $b\times A\times L$. This statement will not be used in the paper. Its proof 
   omitted to save the space follows arguments similar to the proof of Lemma \ref{lemgerm} given below. The author does not 
   know whether the statement of Proposition \ref{clemgerm} holds in full generality, 
   without requiring the irreducibility of all the germs $\Pi_{ab}$ corresponding to some $A$.
    \end{remark}
 
 For the proof of Lemma \ref{lemgerm} we introduce affine coordinates $(x,y)$ centered at $A$ so that $L$ is the $x$-axis.  We fix an arbitrarily small  $c>0$, and for every $t\in a$ we consider the cone $K_t=K_{c,t}\subset\cp^2$ saturated by the lines 
 through $t$ with moduli of azimuths 
 greater than $c$. We denote 
$$K_t^*=K_{c,t}^*=\text{ the image of the cone } K_t=K_{c,t}$$
  \begin{equation} \text{ under the symmetry with respect to the line } T_ta.\label{kt}\end{equation}
 We already know that the cone $K_t^*$ shrinks to $L$, as $t\to A$ (Proposition \ref{plimits}), thus each connected 
 component of the intersection $K_t^*\cap b$ shrinks to $B$. 
 We show (case by case) that for every $t$ close enough to $A$ each one of the latter components is simply connected. 
 Thus, for those $t$ the complement $b_t=b_{c,t}=b\setminus K_t^*$ is connected,  and it is  the whole curve $b$ with 
 small holes deleted; the latter holes shrink to $B$, as $t\to A$. Note that for every $Q\in b_t$ the line $Qt$ reflects from $T_ta$ to a line $l_t$ through $t$ with modulus of azimuth no greater than $c$. Moreover, each line $l_t$ through $t$ with azimuth less than $c$ corresponds to some $Q\in b_t$. Let us localize the analytic set $\Pi_{ab}$ by 
 the inequality $|\az l|<c$ with small  $c>0$. Then for every $t\in a$ close enough to $A$ the preimage of the point $t$ under the projection 
 $\Pi_{ab}\to a$ is a connected holomorphic curve conformally projected onto $b_t\subset b$ that accumulates to the curve $b\times A\times L$, 
 as $t\to A$. This implies the irreducibility of the germ $\Pi_{ab}$. 

 
 The most technical cases of Lemma \ref{lemgerm} are cases (iii) and (iv), when the germs $(a,A)$ and $(b,B)$ are tangent to each other. 
 In the proof of the lemma in those cases we use Proposition \ref{tang-asympt} stated and proved below that concerns 
 the family of tangent lines to the curve $b$. It describes the  
 asymptotic relation between the tangency point and the intersection points of the tangent line with the curve $a$. It will imply that in case 
(iii) with $r_a<r_b$ and in case (iv)  the cone $K_t^*$ contains no tangent line to $b$. This in its turn implies the simple connectivity of the components of the intersection $b\cap K_t^*$. In case (iii) with $r_a\geq r_b$ we study the projections of the components of the
 intersection $K_t^*\cap b$ to the $x$-axis. We show that the projection of each component lies in a disk disjoint from $x(t)$. 
 This together with the Maximum Principle implies that the minimal topological disk $U_t$ containing the component is disjoint from 
 the vertical line $\{ x=x(t)\}$. This together with the Maximum principle, now applied to the 
 projection $U_t\to\cp^1$ from the point $t$ implies that the intersection component under question is simply connected. 
 
 
Now let us pass to the proofs. 

 We consider parametrized curves (germs) and identify them with their parameter spaces (disks in $\cc$). Let $(a,A)$ and $(b,B)$ be 
 distinct tangent irreducible non-linear germs: $A=B=O$, $T_Aa=T_Bb=L$.  Let $(x,y)$ be affine coordinates centered at $O$ such that $L$
  is the $x$-axis. Set 
 $$T_{ab}=\{ (t,\tau)\in a\times b \ | \  t\in T_{\tau}b\}, \ 
 v=x(t), \ u=x(\tau).$$
 The subset $T_{ab}\subset a\times b$ represents a germ of  one-dimensional analytic set at $(O,O)$. 
 We consider its irreducible components  and their  projections to the product 
 $L\times L$: both $t$ and $\tau$ are projected to $L$ along the $y$-axis. Each irreducible component defines two implicit 
 multivalued functions: 
 
 - the function $u=u(v)$, whose graph is the image of the component under the above projection;
 
  - the function $\alpha=\alpha(v)$: the azimuth of the tangent line $T_{\tau}b$. 
  
We normalize the coordinates  so that the curves $a$ and $b$ are graphs of functions 
\begin{equation} b=\{ y=x^{r_b}(1+o(1))\}; \ a=\{ y=\sigma x^{r_a}(1+o(1))\}; \ \sigma\neq0.\label{sigm}\end{equation}
  
 \begin{proposition} \label{tang-asympt} Let $a$, $b$ be two tangent irreducible non-linear 
 germs of analytic curves at a point $O\in\cp^2$. Let the coordinates 
 $(x,y)$, the number $\sigma$, the germ $T_{ab}$  and the functions $v$ and $u$ be as above.  
 Then for every irreducible component of the germ $T_{ab}$  the corresponding 
 implicit  functions $u(v)$ and $\a(v)$ have asymptotic Puiseaux expansions at $0$ of the following possible types;  
 for every given pair $(a,b)$ all the corresponding  asymptotics are realized by appropriate irreducible components: 
 
 Case 1): $r_a>r_b$. Two possible asymptotics for every $(a,b)$: 
 \begin{equation}u=sv(1+o(1)), \ s=\frac{r_b}{r_b-1}, \ \a=r_bs^{r_b-1}v^{r_b-1}(1+o(1)), \label{case31}
 \end{equation}
 \begin{equation} u=sv^{\frac{r_a-1}{r_b-1}}(1+o(1)), \ s=(\frac{\sigma}{r_b})^{\frac1{r_b-1}}, \ 
 \a=\sigma v^{r_a-1}(1+o(1)).\label{case32}\end{equation}

 Case 2): $r_a=r_b=r=\frac pq$, $p,q\in\mathbb Z$ are relatively prime. Then 
 \begin{equation}u=s^qv(1+o(1)), \ \a=rs^{p-q}v^{r-1}(1+o(1)); \ s^{p}(r-1)-r s^{p-q}+\sigma=0.\label{case2}
 \end{equation}
 
 Case 3): $r_a<r_b$. Set $r_g=\frac{p_g}{q_g}$, as above, $g=a,b$. Then 
 \begin{equation}
 u=s^{q_b}v^{\frac{r_a}{r_b}}(1+o(1)), \ s^{p_b}=\frac{\sigma}{1-r_b};  \ 
 \a=r_bs^{p_b-q_b}v^{\frac{r_a(r_b-1)}{r_b}}(1+o(1)).\label{case1}
 \end{equation}
  \end{proposition}
 
 \begin{proof} 
The germ $T_{ab}$  is given by zero set of an analytic function germ on $a\times b$  at 
$(O,O)$ that has the type
 $u^{r_b}(1+o(1))+r_bu^{r_b-1}(v-u)(1+o(1))-\sigma v^{r_a}(1+o(1))$ in the variables $u$ and $v$.
Or equivalently, by an equation 
 \begin{equation} (1-r_b)u^{r_b}(1+o(1))+r_bu^{r_b-1}v(1+o(1))-\sigma v^{r_a}(1+o(1))=0  
 \label{tang-form}\end{equation}
 with the left-hand side being an analytic function of the parameters of the curves $a$ and $b$.
An implicit function $u(v)$ corresponding to an irreducible component of the germ $T_{ab}$ is a solution to (\ref{tang-form}) 
that has Puiseaux expansion without free term. Hence, the restrictions to its graph of the 
  three monomials $(1-r_b)u^{r_b}$, $r_bu^{r_b-1}v$, $-\sigma v^{r_a}$  should satisfy 
  the  following statements as multivalued functions in $v$ after substitution $u=u(v)$:
  
  - at least two of the above monomials  have  lower Puiseaux terms in $v$ with equal powers; 
  we call them {\it principal monomials};  their sum is of smaller order, i.e., it starts with higher terms; 
  
  - the remaining  monomial (if any) should be of smaller order than the principal ones.
  
  In more detail, consider the Newton diagram in $\rr^2$ of the above triple of monomials. That is, take the union of 
  the translation images of the positive quadrant by the vectors $(r_b,0)$, 
  $(r_b-1,1)$, $(0,r_a)$. The Newton diagram is its convex hull. Its {\it edges} are segments in  its boundary that are not contained in 
  the coordinate axes. For every irreducible component of the germ $T_{ab}$ the corresponding 
  principal monomials should lie in the same  edge of the Newton diagram. Vice versa, each edge is realized by an irreducible component. 
  This is a version of a classical observation due to Newton. 
  
 Case 1): $r_a>r_b$. Then the  Newton diagram has  two 
 edges: the segments $[(r_b,0), \ (r_b-1,1)]$ and 
 $[(r_b-1,1), \ (0,r_a)]$. These edges correspond to 
 asymptotics (\ref{case31}) and (\ref{case32}) respectively.
 
   Case 2): $r_a=r_b$. Then there is a unique edge $[(r_b,0), \ (0,r_a)]$,  the three above monomials 
   lie there and are principal.  This implies (\ref{case2}). 
 
 Case 3): $r_a<r_b$. We have one edge $[(r_b,0), \ (0,r_a)]$, the point $(r_b-1,1)$ is in the interior of the Newton diagram,  
 the principal monomials are $(1-r_b)u^{r_b}$ and $-\sigma v^{r_a}$. This implies (\ref{case1}).
Proposition \ref{tang-asympt} is proved.
\end{proof}

   \begin{proof} {\bf of Lemma \ref{lemgerm}.} As it was shown above, for  the proof of the lemma 
    it suffices to prove that for every $t\in a$ close to $A$ 
   each connected component of the intersection $K_t^*\cap b=K_{c,t}^*\cap b$ 
   is simply connected. Let us prove this case by case. To do this, we 
   consider the projection $\nu_t:b\to\oc_t=\mathbb P(T_t\cp^2)$ of the curve $b$ from the point $t$. Note that the intersection $K_t^*\cap b$ is 
   the preimage of a disk  $D(t)=\nu_t(K_t^*)\subset\oc_t$; the symmetry with respect to  the line $T_ta$ sends the disk 
   $D(t)$ to another disk that 
  correspond exactly to the lines through $t$  with moduli of azimuths greater than $c$. Let $(x,y)$ be affine coordinates centered at $A$ with 
  $L$ being the $x$-axis. We identify all the projective lines $\oc_t=\oc$ 
  by translations and introduce the coordinate $z=\frac yx$ on $\oc_A$: $z(L)=0$. This induces coordinate $z$ on each $\oc_t$. 
     
   Case (i): $A\neq B$. Then there exist neighborhoods $U=U(B)\subset b$ and $V=V(A)\subset a$ such that for every $t\in V$ the image 
   $\nu_t(U)\subset\oc$ lies in the unit disk $D_1$: if $\tau\in b$ is close to $B$ and $t\in a$ is close $A$, then the line $t\tau$ is close to $L=AB$. 
   This together with the Maximum Principle applied to the projection $\nu_t:U\to D_1$ 
   implies the simple connectivity of components of the intersection 
   $K_t^*\cap b$. 
   
   Case (ii): $A=B$ but $T_Bb\neq L$. Then for every $t\in a$ close enough to $A$ the cone $K_t^*$  contains no tangent lines to $b$ 
   through $t$, since $K_t^*$  shrinks to the line $L$ transverse to $T_Bb$, as $t\to A$. Therefore, the projection $\nu_t$ 
   of each component of the intersection $K_t^*\cap b$ to the disk $D(t)$ is a branched covering either without critical points, or with 
exactly one critical point of maximal multiplicity. The latter happens exactly when the intersection component under question 
contains $B$ and the latter is a cusp of the curve $b$: this is the critical point. In both cases the component is obviously simply 
connected. 

 Case (iv): $A=B=O$ is a finite point, $T_Aa=T_Bb=L$, and $r_b(2-r_a)<r_a$. We choose $(x,y)$ to be a finite affine chart. 
 Let us show that for every $t\in a$ close to $O$ 
 the cone $K_t^*$ contains no tangent line to the curve $b$, as in Case (ii). Indeed, 
 the azimuths of all the lines  forming the cone $K_t^*$ have uniform asymptotics $O((x(t))^{2(r_a-1)})$, by (\ref{az-fin}). On the other 
 hand, the azimuths of the tangent lines to $b$ through $t$ are of order $(x(t))^{\mu}$, $\mu<2(r_a-1)$. Indeed, 
in the case, when 
 $r_a\geq r_b$, one has $\mu\in\{ r_b-1, \ r_a-1\}$, see (\ref{case31})--(\ref{case2}), hence $\mu\leq r_a-1<2(r_a-1)$. 
 In the case, when $r_a<r_b$, one has $\mu=\frac{r_a(r_b-1)}{r_b}$, by (\ref{case1}), and hence $\mu<2(r_a-1)$: 
 this is equivalent to the 
 inequality $r_b(2-r_a)<r_a$ from the assumption.  Thus, the azimuths of the tangent lines uniformly asymptotically dominate the 
 azimuths of the lines forming the cone $K_t^*$. Hence, the cone $K_t^*$ contains no tangent lines, whenever $t$ is close enough to $O$, 
 and all the components of the intersection $K_t^*\cap b$ are simply connected, as in Case (ii).
 
 Case (iii): $A=B=O$ is an infinite point and $T_Aa=T_Bb=L$. Then the azimuths of the lines forming the cone $K_t^*$ have uniform 
asymptotics $p\frac{y(t)}{x(t)}(1+o(1))=p\sigma(x(t))^{r_a-1}$, where $p\in\{\frac r2, \ 1, \ \frac{r^2}{2r-1}\}$, $r=r_a$, by (\ref{az-inf}); 
here $\sigma$ is the same, as in (\ref{sigm}). 

Subcase (iii a): $r_a<r_b$. Then the azimuths of the tangent lines to $b$ through $t$ are of order $(x(t))^{\mu}$, $\mu=\frac{r_a(r_b-1)}{r_b}
>r_a-1$, see (\ref{case1}). Therefore, the cone $K_t^*$ contains no tangent lines to $b$, whenever $t$ is close enough to $O$, and we are done, as in Case (ii). 

Subcase (iii b): $r_a\geq r_b$.  The intersection points $\tau\in b\cap K_t^*$ satisfy the asymptotic equation
\begin{equation} y(\tau)+p\frac{y(t)}{x(t)}(1+o(1))(x(t)-x(\tau))=y(t).\label{ytt}\end{equation}
Substituting the expressions $y(\tau)=(x(\tau))^{r_b}(1+o(1))$ and $y(t)=\sigma(x(t))^{r_a}(1+o(1))$, $u=x(\tau)$, $v=x(t)$, we get 
\begin{equation} u^{r_b}(1+o(1))-p\sigma v^{r_a-1}u(1+o(1))+\sigma v^{r_a}(p-1+o(1))=0.\label{yttt}\end{equation}

{\bf Claim 1.} {\it Let $r_a\geq r_b$. Then there exists a finite subset $S\subset\cc$ depending only on $r_a$, $r_b$ and $\sigma$ 
such that the projection to the $x$-axis  of the intersection $K_t^*\cap b$ lies in $o(v)=o(x(t))$- neighborhood of the subset 
$Sv\subset\cc$, as $t\to O$.  If there exists a  family $\tau_t$ of intersection points such that 
$x(\tau_t)=v(1+o(1))$ (along a sequence $t_k\to O$), then  $r_a=r_b$ and $y(\tau_t))=y(t)(1+o(1))$.}
%

\begin{proof} In the case, when $r_a>r_b$, all the solutions  $u$ to (\ref{yttt}) are $o(v)$. Indeed, otherwise,  
if there existed a family of solutions 
$u$ to (\ref{yttt}) that is not $o(v)$ along a sequence of points $t\to O$, then the term $u^{r_b}$ in (\ref{yttt}) would have dominated 
the rest of (\ref{yttt}), which is obviously impossible. Let us now treat the case, when $r_a=r_b=r$. 
Consider the three monomials $u^{r}$, $-p\sigma v^{r-1}u$ and $(p-1)\sigma v^{r}$ in the left-hand side of (\ref{yttt}). 
In the case, when $p\neq1$, their Newton diagram consists of one edge $[(r,0),(0,r)]$: the point $(1,r-1)$ lies on this edge. 
 These are the principal monomials in a similar sense, as in the proof of Proposition 
\ref{tang-asympt}. Therefore, if $p\neq 1$, then each solution to (\ref{yttt}) has asymptotics $u=sv(1+o(1))$, with $s$ being a root of polynomial 
depending on $r$, $p$ and $\sigma$, as in the same proposition. The number  $p$ being a three-valued function in 
$r$, we get a triple of  polynomials depending on $r$ and $\sigma$. This proves the claim for the union $S$ of their roots. 
Let now $p=1$: then the third monomial vanishes.  For every family of solutions $u=u(v)$ to (\ref{yttt}) 
one has $u=O(v)$: otherwise, the monomial $u^r$ would have dominated the rest of the left-hand side in (\ref{yttt}) along a subsequence 
$t_k\to O$, which is impossible, as above. Substituting $p=1$ and $u=O(v)$ to (\ref{yttt}) yields that the 
expression  $u^r-\sigma v^{r-1}u$ should be of order $o(v^r)$. This implies that each family of solutions $u=u(v)$ to (\ref{yttt}) 
has asymptotics either $u=o(v)$, or $u=v(s+o(1))$ with $s=\sigma^{\frac1{r-1}}$. This implies the first statement of the claim with 
$S$ consisting of zero and all the numbers $\sigma^{\frac1{r-1}}$. 
Let now there exist a family of solutions $u=v(1+o(1))$ to (\ref{yttt}), i.e., there exist a family of solutions $\tau$ to (\ref{ytt}) with 
$x(\tau)=x(t)(1+o(1))$. Then $r_a=r_b$, by the above discussion, and $y(\tau)=y(t)(1+o(1))$, by (\ref{ytt}). This proves the claim.
\end{proof}  
   
Given a family $U_t$ of connected components of the intersection $K_t^*\cap b$, let us prove their simple connectivity, whenever $t$ 
is close enough to $O$. Let $V_t\subset b$ denote the minimal topological disk containing $U_t$. Without loss of generality we consider 
that there exists an $s\in S$ such that for every $t$ close enough to $O$ 
the projection of the domain $U_t$, and hence $V_t$  to the $x$-axis lies in $o(v)$-neighborhood of the point $sv$, passing to subsequence 
$t_k\to O$ (Claim 1).  Let $s\neq1$. Then the domain $V_t$ does not contain $v=x(t)$, whenever  $t$ is close enough to $O$. 
  Therefore, the image of 
 its projection $\nu_t:V_t\to\oc_t=\mathbb P(T_t\cp^2)$ from the point $t$ is contained in the affine chart $\cc=\oc_t\setminus\{x=x(t)\}$. 
 Hence, $U_t=V_t$ is simply connected, by the Maximum Principle, as in Case (i). Let now $s=1$. We claim that then 
 $U_t$ cannot be tangent to a line through $t$. Indeed,  the $x$- and $y$- coordinates of all the 
 points $\tau\in U_t$ are asymptotically equivalent to those of the point $t$, as $t\to O$, by the claim. 
 Therefore, if there existed  a $\tau\in U_t$ such that 
 the line $t\tau$ is tangent to $b$ at $\tau$, then the azimuth of the tangent line would have been asymptotically equivalent to 
 that of the line $T_ta$: both of them should be equivalent to $r\frac{y(t)}{x(t)}$. 
 Thus, the cone $K_t^*$ would have contained a line with azimuth $r\frac{y(t)}{x(t)}(1+o(1))$, while we already know that 
 the azimuths of all its lines should have  uniform asymptotics $p\frac{y(t)}{x(t)}(1+o(1))$ with $p\in\{\frac r2, \ 1, \ \frac{r^2}{2r-1}\}$, see 
 (\ref{az-inf}).  The latter three possible values of the number $p$ are distinct from $r$, since $r>1$. The contradiction thus obtained 
 proves that in the case under consideration the component $U_t$ is not tangent to a line through $t$, and hence, is simply connected, as in 
 Case (ii). This finishes the proof of simple connectivity of the components of the intersection $K_t^*\cap b$. Together with the 
 discussion following Remark \ref{remiv}, this proves Lemma \ref{lemgerm}. 
 \end{proof} 
 
  \def\mcv{\mathcal V}
  \def\mcc{\mathcal Q}
  
%
 \begin{proof} {\bf of Proposition \ref{clemgerm}.}  The mapping $\psi_a$ contracts the curve $\hb\times A$ to a point $(A,D)$ with 
 some $D\in\hd\cap L$, by Proposition \ref{plimits} and birationality. 
 Fix an arbitrary $D\in\hd\cap L$, and let us show that $\psi_a$ contracts 
 the curve $\hb\times A$ to $(A,D)$. This together with holomorphicity of the mapping $\psi_a$  on finitely punctured curve $\hb\times A$ 
 (isolatedness of indeterminacies) implies uniqueness of the intersection point $D$. Fix 
 a line $l\neq L$  through $A$, a point $Q\in l\setminus A$ and  consider the family of lines $l_t=Qt$ with $t\in \ha$:  $l_t\to l$, 
 as $t\to A$. For every $t\in\ha$ that is not an isotropic tangency point let 
 $l_t^*$ denote the line symmetric to $l_t$ with respect to the tangent line $T_ta$. One has $l_t^*\to L$, by Proposition \ref{plimits}. 
 Fix a family of  intersection points $D_t\in l_t^*\cap\hd$ that tend to $D$, as $t\to A$. Let $B_t\in\hb\cap l_t$ denote the family of points 
 such that $\psi_a(B_t,t)=(t,D_t)$: it exists and is unique by birationality. The family of pairs $(B_t,t)\subset\hb\times\ha$ is obviously an 
 algebraic curve, hence, there exists a limit $B=\lim_{t\to A}B_t\in\hb\cap l$. If $ B\neq A$, then $AB=l\neq L$, $\psi_a$ is holomorphic 
 at $(B,A)$, $\psi_a|_{\hb\times A}\equiv D$ in a neighborhood of the pair $(B,A)$ (Proposition \ref{plimits}), and we are done. 
 Let now $B=A$. The irreducible surface germ $\Pi_{ab}$ contains the germs at $(B,A,L)$
  of analytic  curves  $\gamma=\{(B_t,t,l_t^*) \ | \ t\in a_A\}$  and 
  $\Gamma=B\times A\times\cp^1$, where $\cp^1$ is the space of projective lines through the point $A$. The  germ 
   $\Pi_{ab}$ is smooth in the complement to the curve $\Gamma$. 
 Consider the mapping $\mcc: \Pi_{ab}\to \ha\times\hd$: 
 $\mcc(B',A',L')=\psi_a(B',A')$. This is a well-defined holomorphic mapping on the complement  $\Pi_{ab}\setminus\Gamma$,  
  by birationality and isolatedness of indeterminacies of the mapping $\psi_a$. One has $\mcc(x)\to(A,D)$, as $x\to(B,A,L)$. 
 Indeed, as $x=(B',A',L')\to(B,A,L)$,  set $(A',D')=\mcc(x)$, the points $D'$ tend to the intersection 
 $\hd\cap L$. The germ $\Pi_{ab}$ being irreducible, there are arbitrarily small connected neighborhoods of the point 
 $(B,A,L)$ in $\Pi_{ab}$, and their complements to $\Gamma$ are also connected. 
 For every latter  neighborhood $V$ and each $x=(B',A',L')\in V\setminus\Gamma$ the line $L'$ is close to $L$. 
 Therefore, the corresponding point $D'$ should stay in one and the same local branch intersecting $L$ of the curve $\hd$, as $x$ ranges in 
 $V\setminus\Gamma$. This implies that the limit of the point $D'$ exists and lies in $L$, as   $x\to(B,A,L)$. 
  If $x\to(B,A,L)$ along the curve $\gamma$, then $D'\to D$, by construction. Hence, $\mcc(x)\to(A,D)$, as 
  $x\to(B,A,L)$ in $\Pi_{ab}\setminus\Gamma$. 
  The germ at $(B,A,L)$ of the curve $\hb\times A\times L$ is contained in $\Pi_{ab}$, and $\mcc$ contracts it to the 
  point $(A,D)$, by Proposition \ref{plimits} and the previous statement. 
Or  equivalently,   $\psi_a$ contracts $\hb\times A$ to  $(A,D)$. This proves the proposition.
  \end{proof}

\section{Rationality, property (I)  and coincidence of isotropic tangent lines and opposite mirrors}

In this section and in what follows we consider that  $a$, $b$, $c$, $d$ is a 4-reflective billiard, and no its mirror  is a line. 
 In Subsection 3.1 we prove birationality of neighbor edge correspondence (Lemma \ref{lbirat}) and deduce 
 rationality of mirrors (Corollary \ref{contrat}) and description of the  degenerate quadrilateral orbits with one of the 
 vertices being a cusp with isotropic tangency (Corollary \ref{bircusp}).  In Subsection 3.2 we show that the mirrors have property (I), 
 common isotropic tangent lines, and opposite mirrors coincide: $a=c$, $b=d$.  

\subsection{Birationality of billiard  and rationality of mirrors.}\label{subbirat}
\begin{lemma} \label{lbirat} Let $U\subset\ha\times\hb\times\hc\times\hd$ be the 4-reflective set. 
There exists a unique birational isomorphism $\psi_a:\hat b\times \hat a\mapsto \hat a\times \hat d$ 
such that $\psi_a(B,A)=(A,D)$ for every $ABCD\in U$.  In particular, the algebraic set $U$ is  irreducible, 
the projection $U\to\ha\times\hb$ is birational and  the curves $b$ and $d$ have equal degrees.  
 \end{lemma}
 
 \begin{proof} The algebraic set $U$ is epimorphically projected onto $\ha\times\hb$ (Proposition \ref{comp-set}),  and its projection to 
 $\hb\times\ha\times\hd$ defines 
 an algebraic correspondence $\psi_a:\hb\times\ha\to\ha\times\hd$: $\psi_a(B,A)=(A,D)$ for every $ABCD\in U$. 
 Suppose the contrary to birationality: one of the mappings $\psi_a^{\pm 1}$, say, 
 $\psi_a$ is multivalued on an open set. Then there exists an open subset $V\subset \hb\times\ha$ such that 
for every   $(B,A)\in V$ there exist at least  two distinct quadrilateral orbits $ABC_1D_1$ and 
$ABC_2D_2$ of the billiard $a$, $b$, $c$, $d$, 
where $C_j$ and $D_j$ depend analytically on $(B,A)$; $(C_1,D_1)\neq(C_2,D_2)$ on  $V$. The latter immediately implies that 
$C_1\neq C_2$ and $D_1\neq D_2$ on $V$ (after shrinking $V$). 
Indeed, if $C_1\equiv C_2\equiv C$ but $D_1\neq D_2$ on $V$, then 
$CD_1\equiv CD_2\equiv D_1D_2\equiv AD_1\equiv AD_2$, since the lines $CD_j$ and $AD_j$ are the images of the lines 
$BC$ and $AB$ under 
the symmetries with respect to the lines $T_Cc$ and $T_Aa$ respectively. Thus, $A$, $C$ and $D$ are on the same line for an 
open set of quadrilateral orbits $ABCD$, which is  impossible.  The 
quadrilaterals $C_1D_1D_2C_2$ form an open set of 4-periodic orbits of the billiard 
$c$, $d$, $d$, $c$ with two pairs of coinciding neighbor mirrors, - a contradiction to 
Corollary \ref{lnoncoinc}. This  proves Lemma \ref{lbirat}.
\end{proof}
%
 
 \begin{corollary} \label{contrat} Let $a$, $b$, $c$, $d$ be a 4-reflective algebraic billiard, and the  curve $a$ be not a line. 
 Then $b$ and $d$ are rational  curves. Thus, the mirrors of a 4-reflective algebraic 
 billiard are rational, if none of them is a line. 
 \end{corollary}
 
 As it is shown below, Corollary \ref{contrat} is immediately implied by Lemma \ref{lbirat} and 
 the following well-known theorem from algebraic geometry. It is a part of 
 the Indeterminacy Resolution Theorem  for birational mappings.
 
 \begin{theorem} \label{trat} (implicitly contained in \cite{griff},  p.546 
 of  Russian edition). 
 Let $\psi:S_1\to S_2$ be a birational isomorphism of smooth complex 
 projective surfaces. Then each curve in $S_1$ contracted by $\psi$ to a point in $S_2$ 
 is rational.
 \end{theorem}
%
%
%

 \begin{proof} {\bf of Corollary \ref{contrat}.} It suffices to show that $b$ is rational, 
 by symmetry. There exists an isotropic tangency point $A\in\ha$, since $a$ is not a line. 
 The birational isomorphism $\psi_a:\hb\times\ha\to\ha\times\hd$ from Lemma \ref{lbirat} 
contracts the fiber $\hb\times A$ to a pair $(A,D)\in\ha\times\hd$ with 
$D\in T_Aa$, by Proposition \ref{plimits}, as in the proof of Proposition \ref{clemgerm}. 
Therefore, the  fiber 
$\hb\times A\simeq\hb$ is a rational curve (Theorem \ref{trat}), and thus,  so is $b$. Corollary \ref{contrat} is proved.
 \end{proof}
%

\begin{corollary} \label{bircusp} Let $a$, $b$, $c$, $d$ be a 4-reflective algebraic 
billiard, and $U$ be its 4-reflective set. Let $D\in\hd$ be a cusp with 
isotropic tangent line. Let $\Gamma\subset U$ be a parametrized analytic curve consisting of  quadrilaterals $ABCD\in U$ with fixed $D$ 
and variable $A$ or $C$. Then $B\equiv const$ along $\Gamma$, and $B$ is a cusp of the same degree (see Footnote 4 in Subsection 1.2), 
as $D$.
\end{corollary}

\begin{proof} Let, e.g., $A$ vary along the curve $\Gamma$. 
Fix a quadrilateral $ABCD\in\Gamma$ with $A\neq B,D$ and $A$ being not a marked point ($A\not\equiv B$, by Corollary \ref{lnoncoinc}). 
Then the birational 
transformations $\psi_a^{\pm 1}$ are biholomorphic at $(B,A)$ and $(A,D)$ respectively. The biholomorphicity together with the 
reflection law imply that the intersection index of the variable line $AB$ with the local branch $b_B$ is equal to that of the variable 
line $AD$ with $d_D$, i.e., the degree of the cusp $D$. Thus, it is greater than one, and we have two possibilities: 
either $AB=T_Bb$ and $B$ varies along the 
curve $\Gamma$; or $B\equiv const$ along $\Gamma$ and it is a cusp of the same degree, as $D$, and we are done. 
Suppose the contrary: the former, 
tangency case takes place. Then $B\not\equiv C$ along $\Gamma$ (Corollary \ref{lnoncoinc}), and we can and will consider that $B\neq C$ 
 and the line $T_Bb$ is not isotropic in 
our quadrilateral $ABCD$. Hence,  $ABCD\notin U$, by Corollary \ref{connect} and 
since the tangent line $T_Dd$ is isotropic.  The contradiction thus obtained proves Corollary \ref{bircusp}. 
\end{proof}

\subsection{Property (I) and coincidence of isotropic tangent lines and opposite mirrors}

The main result of the present subsection is the following lemma. 

\begin{lemma} \label{commi} The curves $a$, $b$, $c$, $d$ have property (I)  
and common isotropic tangent lines, and $a=c$, $b=d$. 
\end{lemma}


\begin{proof} Let us first prove property (I) and coincidence of isotropic tangent lines. 
To do this, it suffices to show that for every isotropic tangent line $L$ to any mirror each one of the curves $\ha$, $\hb$, $\hc$, $\hd$ intersects $L$ at a single point. 

Step 1. Fix an isotropic tangent line $L$ to some mirror. Among the tangent branches to $L$ of the mirrors let us fix 
the one that either has infinite tangency point, or  is smooth, or has the maximal possible Puiseaux exponent. 
Let this be, say, the local branch of the curve $a$ at a point $A\in\ha\cap L$. 
The curves $b$, $a$ and $d$ taken in this or inverse order form a reflection-birational triple (Definition \ref{refbir} and Lemma \ref{lbirat}). 
For every $B\in\hb\cap L$ the pair of germs $a_A$ and $b_B$ satisfy some of the conditions (i)--(iv) of Lemma \ref{lemgerm}: 
if $A=B$ is a finite point and $T_Bb=L$, then  inequality  (\ref{puiseq})   
follows from the assumption that either the germ $a_A$ is smooth (hence $r_a\geq2$), or $r_a\geq r_b$. Therefore, the curve $\hd$ intersects the line $L$ at a single point $D$ (Corollary \ref{plemgerm}), 
hence $T_Dd=L$. Analogously, the curve $\hb$ intersects the line $L$ at a unique point, by the above arguments with $b$ and $d$ interchanged. 

Step 2. Let us prove that each one of the curves $\ha$ and  $\hc$ intersects the line $L$ at a unique point. To do this, note that 
the  curves $a$, $d$, $c$ taken in this or inverse order also form a reflection-birational triple, by Lemma \ref{lbirat}. 

Case 1): $D$ is either an infinite point, or a finite point that is not a cusp. Then  
some of the conditions (i)--(iv) of Lemma \ref{lemgerm} holds for $(a,A)$, $(b,B)$ replaced by the local branches $d_D$ and  $a_A$ 
respectively, as in Step 1. Therefore,  the 
curve $\hc$ intersects the line $L$ at a unique point, as in Step 1, and analogously so does $\ha$, by symmetry.  
\def\mck{\mathcal K}

Case 2): $D$ is a finite cusp. Consider a one-parametric 
family $\Gamma$ of quadrilaterals $A'B'C'D\in U$ with the above fixed $D$ and variable $A'\in\ha$. Then $B'\equiv const$ on $\Gamma$, 
and it is a cusp, by Corollary \ref{bircusp}. Thus, for every $A'\in a$ 
   the lines $A'B'$ and $A'D$ are symmetric with respect to the line $T_{A'}a$ (reflection law). Therefore, $a$ is 
   a conic with foci $B'$ and $D$, by Proposition \ref{char-conic}. Hence, it intersects the line $L$ at their unique tangency point $A$. 
   Let us show that $a=c$. The conic $a$ has at least two distinct isotropic tangent 
   lines $l_1$ and $l_2$, set $P=l_1\cap l_2$. Each of them intersects the curve $\hd$ at a unique point, by Step 1. These intersection points 
   are distinct, as are their tangent lines $l_1$, $l_2$, and $P\notin d$, by uniqueness.    Consider the  
   composition of the projection $\pi_d:\hd=\oc\to d$ with the projection $d\to\oc=\cp^1$ from the point $P$. This is a rational 
   mapping $\hd=\oc\to\oc$ that  has two distinct critical points with the maximal multiplicity:  the intersection points 
   $\hd\cap l_j$, $j=1,2$. Therefore, it has no other critical points, and in particular, the curve $d$ has no cusps with non-isotropic tangent 
   lines. Hence, $a=c$, by Corollary \ref{a=c}, thus $c$ is a conic tangent to $L$. 
   
   We have proved that the mirrors have property (I) and common isotropic tangent lines. Recall that a curve with property (I) has no 
   cusps with non-isotropic tangent lines (Propositon   \ref{pisotr} and Corollary \ref{cnform}).  
 This together with Corollary \ref{a=c} implies that $a=c$ and $b=d$. Lemma \ref{commi} is proved.
   \end{proof}

\begin{corollary}\label{cortang} Let $a$, $b$, $a$, $b$ be a 4-reflective complex planar algebraic billiard, and no mirror be a line. Then for every $A\in\ha\setminus b$ 
that is not an isotropic tangency point the collection of 
 lines through $A$ tangent to $b$ (with multiplicities) is symmetric  with respect  to the tangent line $T_Aa$. 
\end{corollary}

\begin{proof} It suffices to prove the statement of the corollary for every $A\in\ha$ lying outside the finite set formed by the 
isotropic tangency points in $\ha$ and the intersection  of the curve $\ha$ with the union of the curve $b$, 
isotropic tangent lines to $b$ and the lines tangent to $b$ at the points of intersection $b\cap c$. 
The statement of the corollary remains valid after passing to 
limits. Fix an $A$ as above. 
For every $B\in\hb$ such that  $A\in T_Bb$ the pair $(A,B)$ lifts to a quadrilateral  $ABCD\in U$ with $AB=BC=T_Bb$ 
(Proposition \ref{comp-set}); one has $A\neq B$ and $B\neq C$, by construction.  
Then $AD=DC$ is the tangent line to $b$ symmetric to $AB$ with respect to  the line $T_Aa$, by Corollary  \ref{connect} and since 
$b$ has no cusps with non-isotropic tangent lines, as at the end of the above proof. This proves the corollary.
\end{proof}

\section{Quadraticity and confocality. End of  proof of Theorem \ref{tclass}}

Recall that we assume that in the 4-reflective billiard under consideration no mirror is a line. 
We have already shown that $a=c$, $b=d$ and they are two distinct rational curves with property (I) 
and common isotropic tangent lines. 
In the next subsection we prove that they are conics. Their confocality will be proved in the second 
subsection. The main tools we use here are Corollaries \ref{cnform},  \ref{cortang} and Proposition \ref{tang-asympt}. 

\subsection{Quadraticity of mirrors}
\begin{theorem} \label{caust-conic}
Let a pair of rational planar curves $a$ and $b$ different from lines form a 4-reflective billiard 
$a$, $b$, $a$, $b$. Then  $a$ and $b$ are conics.
\end{theorem}

Theorem \ref{caust-conic} is proved below. First we show that the mirrors have no cusp (the next lemma). 
Then we split the proof of Theorem \ref{caust-conic} into three cases (Proposition \ref{prop-cases}), 
two of them will be treated in  separate propositions. 

\begin{lemma} \label{lem-cusp} In the conditions of Theorem \ref{caust-conic} none of the curves 
$a$ and $b$ has cusps.
\end{lemma}

In the proof of this lemma we use the following proposition.

\begin{proposition} \label{pr-cusp} 
 Let a rational curve  in $\cp^2$ with property (I) have a cusp. Then it has 
either only one cusp, or two cusps. In the latter case its local branches 
at the cusps have distinct degrees.
\end{proposition}

\begin{proof} The curve under consideration has one of the normal forms (\ref{form-pd}) 
or (\ref{form-pq}) in appropriate isotropic coordinates on the finite plane $\cc^2$. Consider, 
e.g.,  case (\ref{form-pd}): then $p$ and $d$ are relatively prime. 
Suppose the contrary: the curve has cusps both at the origin and 
at the infinity, and the corresponding local branches have equal degrees. The latter degrees 
are equal to $p>1$ and $d-p$ respectively. Hence, 
$p$ and $d=2p$ are not relatively prime, -- a contradiction. Case (\ref{form-pq}) is treated 
analogously.
\end{proof}

\begin{proof} {\bf of Lemma \ref{lem-cusp}.} 
Suppose the contrary: say, the curve $b$ has a cusp $B$, let us denote its lifting to $\hb$ also by $B$. Consider an irreducible 
algebraic curve $\Gamma$ consisting of 
quadrilaterals $ABCD\in U$ with fixed $B$ and variable $A$. Then $D\equiv const$ on $\Gamma$, and it is a cusp of the same degree, as 
$B$, by Corollary \ref{bircusp}. Let us show that $B\neq D$.  
The contradiction thus obtained to Proposition \ref{pr-cusp} will prove the lemma. 
%
%
The curve $b$ has one of the normal forms (\ref{form-pd}) or 
(\ref{form-pq}), and each its cusp is either an isotropic point at infinity, or the origin. 
If $B$ is an isotropic point at infinity, then the lines $AD$, which are symmetric to $AB$ with respect to the line $T_Aa$, 
should pass through 
the  other isotropic point at infinity, and the latter obviously coincides with $D$; thus $B\neq D$. 
Otherwise, $B$ is the origin in the normal form coordinates: then one has case (\ref{form-pd}), and the curve $b$ 
is tangent to the infinity line. If $B=D$, then 
for every $A$ the line $AB$ is orthogonal to the tangent line $T_Aa$. Hence $a$ 
is a complexified circle, and thus, is tangent to finite lines at both isotropic points at infinity. 
Therefore, the infinite tangent line to the curve $b$ is transverse to $a$, - a contradiction 
to the fact that $a$ and $b$ have common isotropic tangent lines. Hence, $B$ and $D$ are 
distinct cusps of equal degrees, -- a contradiction to Proposition \ref{pr-cusp}. 
This proves Lemma \ref{lem-cusp}.
\end{proof}  

\begin{proposition}  \label{prop-cases} Let $a$, $b$ be distinct rational curves with property (I),  
common isotropic tangent lines and no cusps. Then one of the following holds:

Case 1): some isotropic line is tangent to  $a$, $b$ at distinct finite points;

Case 2): $a$, $b$ are  tangent at  their common finite isotropic tangency point; 

Case 3): $a$ and $b$ are tangent to each other at both isotropic points at infinity; then they are conics.
\end{proposition} 

\begin{proof} Let us first assume that one of the curves, say $a$ has a finite isotropic tangency point $A_0$, set 
$L=T_{A_0}a$. Then $b$ is tangent to $L$ at some point $B_0$. If $A_0=B_0$, then we have 
case 2). Otherwise, $B_0\neq A_0$, hence $B_0\notin a$ (property (I) of the curve $a$). 
In this case $B_0$ cannot be infinite. Indeed, otherwise, $B_0$ would be an 
isotropic point at infinity and $L$ would be the unique tangent line to $b$ 
 through $B_0$ (property (I) of the curve $b$). But there is another  tangent line $L^*\neq L$  to $a$ 
 through $B_0$, since $B_0\notin a$, $a$ has no cusps and by Riemann--Hurwitz Formula for the projection 
 $a\to\cp^1$ from the point $B_0$. The line $L^*$ is isotropic and not tangent to $b$, by construction and property (I),  
 -- a contradiction to the coincidence of isotropic tangent lines of the curves $a$ and $b$. 
 Thus, we have case 1).  
If each one of the curves $a$ and $b$ has no finite  isotropic tangency point, then one has case 3) (Proposition  \ref{pisotr} and 
Corollary \ref{cnform}). 
Then the curves are conics, by the same corollary and absence of cusps. 
This proves the proposition.
\end{proof}

\begin{remark} \label{rdeg}
For every $g=a,b$ and every  isotropic tangency 
point $G\in g$ (which is not a cusp by Lemma \ref{lem-cusp})  
the Puiseaux exponent $r_g=r_g(G)$ equals the degree of the curve $g$, by property (I). 
\end{remark}

%

\begin{proposition} \label{reg-dist} Let a pair of distinct curves $a$ and $b$ that are not lines 
form a 4-reflective billiard $a$, $b$, $a$, $b$. Let there exist an isotropic 
line $L$ that is  tangent to $a$  and $b$ at distinct finite points.  Then $a$ and $b$ are conics.
 \end{proposition}
 
 \begin{proof} By symmetry, it suffices to prove that $a$ is a conic, i.e., $r_a=2$ (Remark \ref{rdeg}). 
 Let $A_0\neq B_0$ be respectively the tangency points of the line $L$ with $a$ and $b$. 
 Then $A_0\notin b$ (property (I) of the curve $b$).  
 There exists another  line $l\neq L$ through $A_0$ that is tangent to $b$, as in the above proof. 
 Fix isotropic coordinates $(x,y)$ on the finite plane $\cc^2$ with the origin at $A_0$, the $x$-axis  $L$ and $x(B_0)=1$. 
 For every $t\in a$ close to $A_0$ there exists a line $l_t^*$ tangent to $b$ through $t$  that tends to $l$, as $t\to A_0$. 
 Its image $l_t$ under the symmetry with respect to $T_ta$ tends to $L$, and 
 $\az(l_t)=O(v_t^{2(r_a-1)})$, $v_t=x(t)$ (Proposition \ref{tang-asympt1}).  On the other hand, the line $l_t$ 
 should  be tangent to $b$ (Corollary \ref{cortang}), and so is $L$. Therefore, the intersection point $B_t=l_t\cap L$ tends to a finite point $B_0\neq A_0$. Thus, $l_t$ is the line through the points $B_t=(\beta_t,0)$, $\beta_t\to x(B_0)=1$, and 
  $t=(v_t,\sigma v_t^{r_a}(1+o(1)))$, $\sigma\neq0$. Hence, the azimuth $\az (l_t)$ is of order $-y(t)\simeq-\sigma v_t^{r_a}$, 
 but on the other hand, it is $O(v_t^{2(r_a-1)})$. Thus,  
 $2(r_a-1)\leq r_a$ and $r_a=2$. Proposition \ref{reg-dist} is proved.
 \end{proof}
%

\begin{proposition} \label{tang-con} Let a pair of distinct curves $a$ and $b$ that are not lines 
form a 4-reflective billiard $a$, $b$, $a$, $b$. Let they be tangent to each other at their common finite isotropic tangency 
point $O$. Then they are conics.
\end{proposition}

\begin{proof}  Let $(x,y)$ be 
affine coordinates on the finite plane with the origin at $O$, and the tangent line $L=T_Oa=T_Ob$ be the $x$-axis. Let $r_a$ and $r_b$ be the 
Puiseaux exponents at $O$ of the curves $a$ and $b$ respectively (degrees, see Remark \ref{rdeg}), $r_a\leq r_b$.  Suppose the contrary: $r_b\geq3$. 
 Then there exists 
another  line $l\neq L$ through $O$ that  is tangent to $b$, since the projection $b\to\cp^1$ from $O$ has degree $r_b-1\geq2$, $b$ 
has no cusps  and by Riemann--Hurwitz Formula. 
For every $t\in a$ close to $O$ let $l_t^*$   be the tangent line to $b$ through $t$ such that $l_t^*\to l$, as $t\to O$.  
Its image  $l_t$ under the  symmetry with respect to the line $T_ta$ is also tangent to $b$, tends to $L$, and 
$\az (l_t)=O(v_t^{2(r_a-1)})$, $v_t=x(t)$, as in the above proof. On the other hand, the latter azimuth should be a quantity of order 
$v_t^{\nu}$, $\nu=\frac{r_a(r_b-1)}{r_b}$, by Proposition \ref{tang-asympt}. This is impossible, since $\nu<r_a\leq 2(r_a-1)$.  Proposition \ref{tang-con} is proved.
 \end{proof}

\begin{proof} {\bf of Theorem \ref{caust-conic}.}  Lemma \ref{lem-cusp} shows that the curves $a$ and $b$ have no cusps. 
 Proposition \ref{prop-cases} lists all the possible
 cases 1)-3).  Propositions \ref{reg-dist} and \ref{tang-con}  respectively show that 
 in cases 1) and 2) the curves are conics. In case 3) the curves are automatically conics. 
  This proves Theorem \ref{caust-conic}.
 \end{proof}

\subsection{Confocality}

We have shown  that $a$ and $b$ are distinct conics with common isotropic tangent lines. 
Here we prove that they are confocal, by using confocality criterion given by Lemma \ref{propconf}. 
This will finish the proof of Theorem \ref{tclass}.

Case 1): one of the conics, say $a$ is transverse to the infinity line 
(i.e., a transverse hyperbola). Then  $b$ is also a transverse hyperbola confocal to $a$, since it has the same isotropic tangent lines
and by  Lemma \ref{propconf}, case 1).

Case 2): both $a$ and $b$ are tangent to the infinity line, and $a$ is tangent to it at a non-isotropic point $A_0$. Then 
$b$ is also tangent to it at a non-isotropic point $B_0$, by the coincidence of isotropic tangent lines.

\medskip

{\bf Claim 1. } $A_0=B_0$.

\begin{proof} 
Suppose the contrary: $A_0\neq B_0$. Let $I_{1,2}=(1:\pm i:0)$ be the isotropic points at infinity. Then $A_0$, $B_0$, $I_1$, $I_2$ are  
four distinct points on the infinity line.  
 Let us choose an affine chart  $(x,y)$ with the origin  $I_1$, the $x$-axis being  the infinity line with 
 $x=z$ there.  
 Fix an arbitrary point $A\in a$ close to $A_0$. Let $p_A$ denote the tangent line to $b$ through $A$ at a point close to 
 $B_0$: $p_A$ tends to the infinity line ($x$-axis), as $A\to A_0$. Let $p^*_A$ denote the  line  symmetric to $p_A$ with 
 respect to the line $T_Aa$. 
 
 \medskip
 
 {\bf Claim 2.} {\it The line $p_A^*$  tends to the infinity line, as $A\to A_0$.}
 
 \begin{proof} Let $x_0$, $x_1$, $x_2$  denote respectively the $x$-coordinates of the points of intersection of the $x$-axis with the lines 
 $p_A$, $T_Aa$ and $p_A^*$. Set $u=x(A_0)$, $v=x(B_0)$. 
 One has $x_0x_2=x_1^2$ (reflection law and Proposition \ref{tang-asympt1}),  and $x_0\to v$, 
 $x_1\to u\neq v$, as $A\to A_0$.  Therefore, $x_2\to w=\frac{u^2}v\neq u$. Thus, $p_A^*$ is the line through the points $A$ and $(x_2,0)$, 
  one has $(x_2,0)\to (w,0)\neq A_0$, as $A\to A_0$. Hence, $p_A^*$ tends to the $x$-axis, i.e., $\oc_{\infty}$. 
 This proves the claim. 
 \end{proof} 

Thus, the  line $p_A^*$ should be a tangent line to $b$ that tends to the infinity line, by Corollary \ref{cortang} and the above claim. 
There are two tangent lines to $b$ through the point $A_0$: the infinity line and a finite line, since $A_0\notin b$. 
Thus, there are two tangent lines to $b$ through $A$: the line $p_A$ and a line with a finite limit, as $A\to A_0$. Therefore,  
$p_A=p_A^*$, thus $p_A$ is orthogonal to $T_Aa$ for all $A$ close enough to $A_0$. In other words, every tangent line to $b$ close 
to the infinity line is orthogonal to the conic $a$ at both their intersection points. Hence, this is true for all the tangent lines to $b$, by 
algebraicity. Thus, 
 for every point $A\in a$ there is only one tangent line to the conic $b\neq a$ through $A$, which is impossible. This proves Claim 1. 
\end{proof} 

Thus, the conics $a$ and $b$ have common isotropic tangent lines and are tangent to the infinity line at the 
same non-isotropic point, by 
Claim 1. This  together with Lemma \ref{propconf} 
implies that $a$ and $b$ are confocal. 

Case 3): $a$ is tangent to the infinity line at an isotropic point, say  $I_1$. Then $b$ is also tangent to the infinity line at $I_1$, since 
the conics $a$ and $b$ have common isotropic tangent lines. Recall that $a\neq b$. 

\medskip

{\bf Claim 3.} {\it The conics $a$, $b$ are tangent  at $I_1$ with at least triple contact.}
\medskip

\begin{proof} 
Let   $(x,y)$ be affine coordinates with the origin $I_1$, the $x$-axis $\oc_{\infty}$ with $x=z$ there such that  
$b=\{ y=x^2(1+o(1))\}$, as $(x,y)\to(0,0)$. Then $a=\{ y=\sigma x^2(1+o(1))\}$, $\sigma\neq0$. 
We have to show that $\sigma=1$. 
For every $t\in a$ close to $I_1$  set $v_t=x(t)$. There are two tangent lines $p_{1,t}$ and $p_{2,t}$ to $b$ through $t$, and 
they tend to the $x$-axis, i.e., the infinity line, 
as $t\to I_1$. Their  azimuths have asymptotics $2s_jv_t(1+o(1))$, $j=1,2$, where $s_j$ are the roots of the quadratic equation 
 $s^2-2s+\sigma=0$, and their tangency points with $b$ have $x$-coordinates $u_t=s_jv_t(1+o(1))$ (Proposition \ref{tang-asympt}). 
 Therefore, the  latter tangent lines intersect the $x$-axis at points with coordinates $x_j=\frac{s_j}2v_t(1+o(1))$, while the tangent line 
 $T_ta$ intersects it at a point with coordinate $x_0=\frac{v_t}2(1+o(1))$ (Proposition \ref{pxt}). The pair of lines $p_{j,t}$ is symmetric with 
 respect to the line $T_ta$ (Corollary \ref{cortang}). The symmetry cannot fix each  $p_{j,t}$ for all $t$, 
 since otherwise either two distinct lines  $p_{1,t}$ and $p_{2,t}$ would be orthogonal to $T_ta$, or one of them 
 would be tangent to  both $a$ and $b$ for all $t$; none of the latter is possible.  Thus, the 
 symmetry permutes the above lines. Hence, $x_1x_2=x_0^2$ (Proposition \ref{tang-asympt1}), and thus, $s_1s_2=1$. On the other hand, 
 $s_1s_2=\sigma$, thus, $\sigma=1$. The claim is proved.
 \end{proof} 
 
   Claim 3 together with Proposition \ref{ptransl}  imply that in our case  the conics $a$ and $b$ 
 are obtained one from the other by translation of the finite plane. This together with the coincidence of their finite isotropic tangent lines 
 and Lemma \ref{propconf} implies that $a$ and $b$ are confocal. Theorem \ref{tclass} is proved.
 
 \section{Degenerate orbits of billiard on conics}
 Here we consider a 4-reflective billiard $a$, $b$, $a$, $b$ on pair of distinct confocal conics $a$, $b$ and prove the following 
 classification of degenerate quadrilateral orbits. Note that  smooth mirrors  coincide with their normalizations.
 
 \begin{theorem} \label{irrdeg} Let $U\subset a\times b\times a\times b$ be the 4-reflective set, and let $U_0\subset U$ be the subset of 
 4-periodic orbits. The complement $U\setminus U_0$ is the union of the following  algebraic curves:
 
 (i) the curve $\mct_a$ consisting of quadrilaterals $ABCB\in a\times b\times a\times b$ with $A,C\neq B$,  $AB=T_Aa$, $CB=T_Ca$ 
 and all the single-point quadrilaterals $AAAA$, $A\in a\cap b$; $AB\not\equiv CB$ on each component of the curve $\mathcal T_a$; 
 
 (ii) the curve $\mct_b$ defined analogously with $a$ and $b$ interchanged;
 
 (iii)  for every isotropic tangent line $L$ to $a$ and $b$ at points $A_0$ and $B_0$ respectively the curve $\Gamma_{ab}(L)$ 
 consisting of quadrilaterals $A_0B_0AB$ with variable $A$ and $B$; the curve $\Gamma_{ab}$ is rational, and the 
 natural projections $\Gamma_{ab}\to a$, $\Gamma_{ab}\to b$ to the positions of the variable vertices $A$  and $B$ are bijective;
  
 (iv)  three other  rational curves $\Gamma_{bc}(L)$, $\Gamma_{cd}(L)$, $\Gamma_{da}(L)$ 
 defined analogously to $\Gamma_{ab}(L)$ and consisting of quadrilaterals $AB_0A_0B$, $ABA_0B_0$ and $A_0BAB_0$ respectively with variable $A$ and $B$.   
  
  Each one of the curves $\mct_a$ and $\mct_b$ is 
  
  - elliptic, if $a$ and $b$ are not tangent to each other; 
  
  - rational with one transverse self-intersection, if $a$ and $b$ are quadratically tangent at a unique point;
  
  - rational with one cusp of degree two, if $a$ and $b$ are tangent with triple contact;
  
  - a union of two smooth rational curves, if $a$ and $b$ are tangent at two distinct points.
 \end{theorem}
  
 \begin{proof} Fix a quadrilateral $ABCD\in U\setminus U_0$. Let us show that it lies in the union of the above curves. 
  If its vertices are not isotropic tangency points and every two neighbor vertices are distinct, then it lies 
  in $\mct_a\cup\mct_b$, by Corollary \ref{connect}. 
 If it has an isotropic tangency vertex, then some its neighbor is also 
 an isotropic tangency vertex (reflection law and coincidence of isotropic tangent lines of the conics $a$ and $b$). 
 Therefore, given an isotropic line $L$ tangent to $a$ and $b$ at points $A_0$ and $B_0$ respectively, there exists a rational 
 curve $\Gamma_L=\Gamma_{ab}(L)$  of quadrilaterals $A_0B_0AB\in U\setminus U_0$ with variable $A$ (see the proof of Corollary  \ref{contrat}). Let us show that $B$ also varies along $\Gamma_L$. Indeed, in the contrary case, when $B\equiv const$ on $\Gamma_L$, one would have 
  either $B\equiv A_0$, or the variable line $AB$ is reflected to the constant  line $BA_0$ (which is thus isotropic and coincides with 
  $T_Bb$). In both cases $T_Bb=L$, by coincidence of isotropic tangent lines of the conics $a$ and $b$. Hence, $B\equiv B_0$ on 
  $\Gamma_L$. Therefore, $AB_0$ is orthogonal to $T_Aa$ for every $A\in a$, and $a$ is a complex circle centered at $B_0$. 
By confocality, $b$ is also a complex circle centered at $B_0\in b$, which is obviously impossible. The contradiction thus obtained proves 
that both $A$ and $B$ vary along the curve $\Gamma_L$. The projection of the curve $\Gamma_L$ to the position of either $A$, or $B$ is 
bijective, by birationality (Lemma \ref{lbirat}). 

 A priori, $ABCD$ may have yet another degeneracy: 
 coinciding pair  of neighbor vertices  that are not isotropic tangency points. First we show that the quadrilaterals in $U$ with the latter 
 degeneracy form a finite set of points  $AAAA\in \mct_a\cap\mct_b$, $A\in a\cap b$. Then we prove the last statements of the theorem 
 on the structure of the curves $\mct_a$ and $\mct_b$. 
 
 \begin{proposition} \label{aaaa} 
 Let $A=B\in a\cap b$ be an intersection point such that the line $T_Aa$ is not isotropic. Then the line $T_Bb$ is also 
 not isotropic, and the 4-reflective set $U$ contains no algebraic curve of quadrilaterals $AACD$ with fixed $A=B$ and variable $C$ or $D$. 
 \end{proposition}
 
 \begin{proof} The non-isotropicity of the tangent line $T_Bb$ follows immediately from the coincidence of isotropic tangent lines to $a$ 
 and $b$ (confocality). The line $T_Bb$ is transverse to $a$,  since each common tangent line to $a$ and $b$ is isotropic 
 (see the proof of Corollary \ref{tan-con}). Suppose the contrary to the second statement: there exists an algebraic curve $H\subset U$ consisting of 
 quadrilaterals $AACD$, say, with variable $C$. Let $C^*$ denote the point of intersection $a\cap T_Ab$ distinct from $A$. Then $H$ contains 
 a quadrilateral $x=AAC^*D^*$ with $AC^*=T_Ab$ and $C^*\neq A$. The  quadrilateral $x$ may be deformed to a 
  quadrilateral $A'B'C'D'\in U$ with $B'C'=T_{B'}b$ and $A'\neq B',C'$, -- 
 a contradiction to Corollary \ref{connect}. The proposition is proved.
 \end{proof}
 
 Let $A$ be as in Proposition \ref{aaaa}. The proposition immediately implies that  
 the birational isomorphisms $\psi_a$, $\psi_b$ and their inverses are biholomorphic at 
 $(A,A)$. Let us show that they send $(A,A)$ to $(A,A)$. Indeed, take a point $A'\in a$, the two tangent lines to $b$ through $A'$ and 
the tangency points $B'$ and $D'$. One has $A'B'A'D'\in\mct_b$, and $B',D'\to A$, as $A'\to A$. This together with biholomorphicity 
and Lemma \ref{lbirat} 
implies that $\psi_a(A,A)=\psi_b^{-1}(A,A)=(A,A)$, $AAAA\in \mct_b$ and $AAAA$ is the only quadrilateral in $U$ with a pair of neighbor 
vertices coinciding with $A$. Similarly, $AAAA\in\mct_a$. 

Let us prove the last statement of Theorem \ref{irrdeg} for the curve $\mct_a$; the case of the curve $\mct_b$ is symmetric. 
In the case, when the conics $a$ and $b$ are not tangent, the ellipticity of the curve $\mct_a$  is classical, see 
\cite{grifhar}. Namely, the projection $q:\mct_a\to b\simeq\oc$: $ABCB\mapsto B$ is a double covering over $\oc$, branching over the 
points of intersection $a\cap b$. If the conics are not tangent, then we have four distinct branching points, thus the curve $\mct_a$ 
is elliptic by Riemann--Hurwitz Formula.  The cases of tangent confocal conics are described by Corollary \ref{tan-con}. 
In the case, when $a$ and $b$ are quadratically tangent at some point $O$, 
the curve $\mct_a$ has two transversely 
intersected local branches through the point $(O,O)$. 
This easily follows from Proposition \ref{tang-asympt}, see (\ref{case2}). Therefore, the projection $q$ 
does not branch at the double intersection $O$ and has only two branching points; thus the curve $\mct_a$ is rational by Riemann--Hurwitz 
Formula. The case of two tangency points is treated analogously: the covering does not branch at all and  has two 
univalent sheets bijectively projected to $\oc$. In the case of triple tangency we have two distinct branching points: the point $O$ and 
another branching point. The implicit function $A=A(B)$ defined by the curve $\mathcal T_a$ is double-valued in the neighborhood of the point 
$O$, and its  both branches 
have unit derivative at $O$, by (\ref{case2}). This easily implies that the quadrilateral $OOOO$ corresponds to a degree two 
cusp  of the rational curve $\mct_a$.   Theorem \ref{irrdeg} is proved.
\end{proof}

 \section{Application: classification of 4-reflective real algebraic planar pseudo-billiards}
 
 \def\rp{\mathbb{RP}}
 \def\ur{U_{\mathbb R}}
 \def\uor{\ur^0}
 
 Here by {\it real analytic curve} we mean a curve  $a\subset\rp^2$ analytically parametrized by either $\mathbb R$, or $S^1$ that is not the 
 infinity line.  If a  curve $a$ has singularities (cusps or self-intersections),  
 we consider its maximal real analytic extension $\pi_a:\ha\to a$, where $\ha$ is either $\mathbb R$, or $S^1$, see 
 \cite[lemma 37, p.302]{gk2}. The parametrizing curve $\ha$ will be called here the {\it real normalization.} 
 The affine plane $\rr^2\subset\rp^2$ is equipped with Euclidean metric. 
 
 \begin{definition} A {\it real planar analytic (algebraic) pseudo-billiard} is  a collection of $k$ real irreducible 
 analytic (algebraic) curves $a_1,\dots,a_k\subset\rp^2$. 
 Its {\it $k$-periodic   orbit}  
  is a $k$-gon $A_1\dots A_k$, $A_j\in a_j\cap\rr^2$, such that for every $j=1,\dots,k$ one has $A_j\neq A_{j\pm1}$,
 $A_jA_{j\pm1}\neq T_{A_j}a_j$ and the lines $A_jA_{j-1}$, $A_jA_{j+1}$ are symmetric with respect to the tangent line $T_{A_j}a_j$. 
 The latter means that for every $j$ the triple  $A_{j-1}$, $A_j$, $A_{j+1}$ and the line $T_{A_j}a_j$ satisfy either usual, or skew 
 real reflection law. For brevity, we then say that usual (skew) reflection law is satisfied at the vertex $A_j$. 
 Here we set $a_{k+1}=a_1$, $A_{k+1}=A_1$, $a_0=a_k$, $A_0=A_k$. A real pseudo-billiard is called 
  {\it $k$-reflective,} if it has  an open set (i.e., a two-parameter family) of $k$-periodic  orbits.  The interior points of the set of $k$-periodic 
 orbits will be called {\it $k$-reflective orbits.} 
   \end{definition}
 
  \begin{remark} \label{rpseudo} The complexification of every real $k$-reflective planar analytic pseudo-billiard is a $k$-reflective analytic 
  complex planar billiard. 
 \end{remark}
 
%
 
 \begin{theorem} \label{class}  A   real planar algebraic pseudo-billiard  $a$, $b$, $c$, $d$  is   4-reflective, 
 if and only if it has one of the following types:
 
 1) $a=c$ is a line,  the curves $b,d\neq a$ are symmetric with respect to it;
 
 2) $a$, $b$, $c$, $d$ are distinct lines  through the same point $O\in\mathbb{RP}^2$, 
the line pairs $(a,b)$, $(d,c)$ are transformed one 
 into the other by  rotation around $O$ (translation, if $O$ is an infinite point); 
 
 3) $a=c$, $b=d$ and they are distinct confocal conics: either ellipses, or hyperbolas, or ellipse and hyperbola, or parabolas. 
 
 In every 4-reflective orbit the  reflection law at each pair of opposite vertices is the same; it is skew for at least one opposite vertex pair. 
 \end{theorem}

 {\bf Addendum 1.} {\it In Theorem \ref{class},  case 1) each 4-reflective orbit   $ABCD$  has the same type, as at Fig.1: 
 it is symmetric with respect to the line $a=c$ and the reflection law  is skew at $A$ and $C$. In the subcase, when 
 $b$, $d$ are lines parallel to $a$, 
 the reflection law is usual at $b$ and $d$.  In the subcase, when $b=d$ is a line orthogonal to $a$,  
the orbits are rhombi symmetric with respect to $a$ 
 and $b$: the reflection law is skew at each vertex, see Fig.11a). If none of the  latter subcases holds, then the billiard has both types of 
 4-reflective orbits $ABCD$: with usual reflection law at $B$, $D$ and with skew one.} 
 
 \medskip
 
 {\bf Addendum 2.} {\it Let a 4-reflective pseudo-billiard have type 2).  
 If the lines $a$, $b$, $c$, $d$  are parallel and $b$, $d$ lie between the lines $a$ and $c$, then  for every 4-reflective orbit 
 the reflection law is usual at $a$, $c$ and skew at $b$, $d$. Otherwise, 
there are three types of 4-reflective orbits: all the reflection laws are skew; usual reflection law at $a$, $c$ and skew at $b$, $d$; 
vice versa, see Fig.2.}

\medskip

%
{\bf Addendum 3.} {\it Let a 4-reflective pseudo-billiard $a$, $b$, $a$, $b$ have type 3). In the case, when $a$ and $b$ are ellipses, 
all the 4-reflective orbits have the same reflection laws, 
as at Fig.4: two usual ones and two skew ones.  In the case, when $a$ and $b$  are either 
 hyperbolas, or ellipse and hyperbola, or parabolas, all the possible reflection law combinations are given in the next figures, up to 
 symmetries with respect to the common  symmetry lines of the  conics $a$ and $b$, renaming 
 opposite mirrors and cyclic renaming of mirrors.}
 
 \begin{figure}[ht]
  \begin{center}
   \epsfig{file=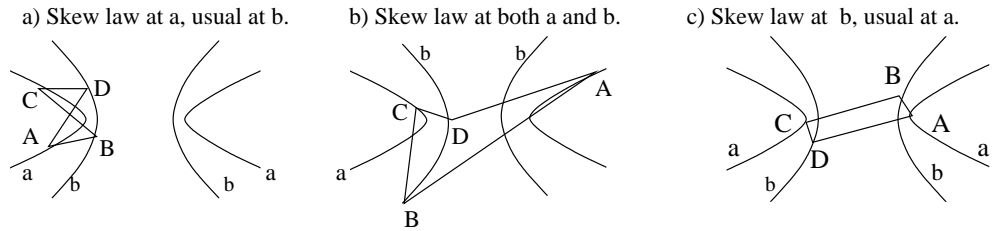}
    \caption{4-reflective  orbits on confocal hyperbolas: three types}
  \end{center}
\end{figure}

\begin{figure}[ht]
  \begin{center}
   \epsfig{file=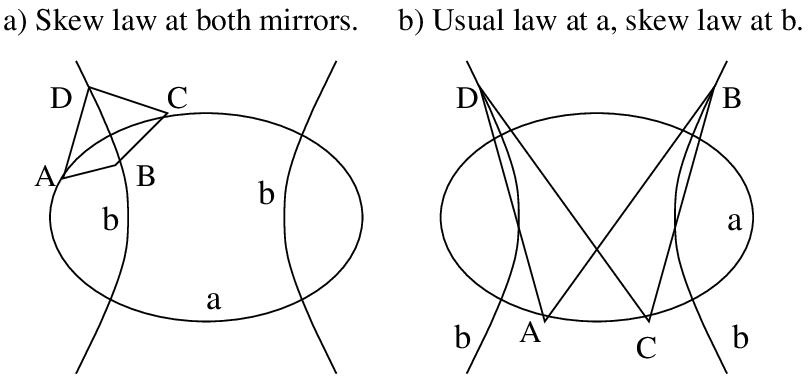}
    \caption{4-reflective orbits on confocal ellipse and hyperbola: two types}
  \end{center}
\end{figure}

\begin{figure}[ht]
  \begin{center}
   \epsfig{file=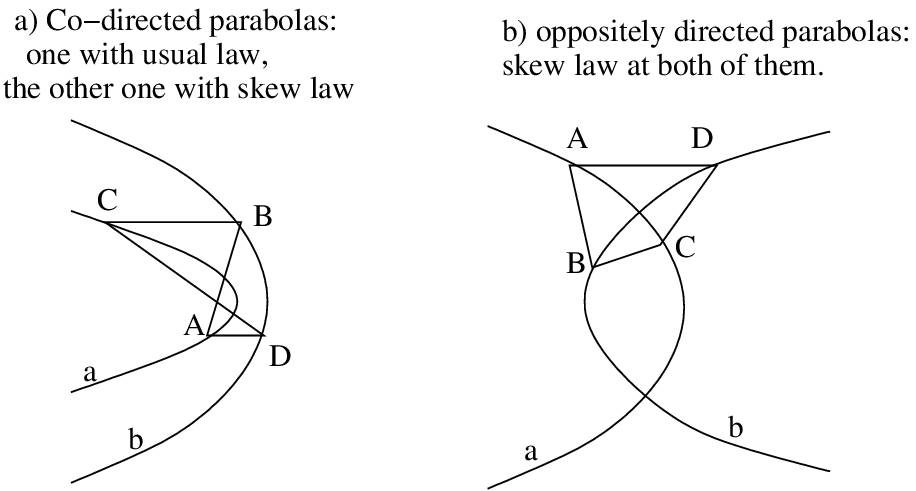}
    \caption{4-reflective orbits on confocal parabolas: one type}
  \end{center}
\end{figure}

\begin{remark} The main result of paper \cite{gk2} (theorem 2) concerns  usual real planar billiards with piecewise-smooth boundary; 
the reflection law is usual. 
It says that the set of quadrilateral orbits has measure zero. In the particular case of billiard with piecewise-algebraic boundary this 
statement follows from the last statement of Theorem \ref{class}. 
\end{remark}

\begin{proof} {\bf of  Theorem \ref{class} and its addendums.} 
The pseudo-billiard $a$, $b$, $c$, $d$ under question being 4-reflective, its complexification is 4-reflective 
(Remark \ref{rpseudo}). This together with Theorem \ref{tclass} 
implies that it has one of the above types 1)--3) (up to  cyclic renaming of the mirrors). Let us show that 
each type is realized and prove the addendums. 

{\bf Case of type 1).} Each 4-reflective orbit $ABCD$ is symmetric with respect to the line $a=AC$: 
the lines $AB$ and $AD$ are symmetric, the lines $CB$ and $CD$ are symmetric; hence, the 
intersection points $B=AB\cap CB$ and $D=AD\cap CD$ are symmetric. The statements of Addendum 1 on the reflection law follow  
immediately from symmetry. 

{\bf Case of type 2).} The 4-reflectivity  was proved in  
Example \ref{ex-lines2}; the proof applies to the real case. Let us prove  Addendum 2. 

Subcase 2a): $a$, $b$, $c$, $d$ are parallel lines, say horizontal. Then the line pair $(a,b)$ is sent to $(d,c)$   
 by translation. Hence, the lower and upper lines 
are opposite mirrors, say $a$ and $c$,  the other intermediate lines $b$ and $d$ are opposite, 
the reflection law at $a$ and $c$ is usual, and that at $b$ and $d$ is skew. 

Subcase 2b): the mirrors pass through the same finite point $O$. 

\medskip

{\bf Claim 1.} {\it The lines $b$ and $d$ punctured at $O$ are contained in a union of two opposite quadrants of the 
complement $\rr^2\setminus(a\cup c)$.}

\begin{proof} Suppose the contrary. Then  two rays $R_b$ and $R_d$  of the lines $b$ and 
$d$ respectively lie in the same half-plane with respect to the line $a$ and are separated by a ray $R_c$ of the line $c$; all the rays 
are issued from the point $O$. This implies that the rotation around the point $O$ sending the ray $R_c$ to $R_b$ cannot send $R_d$ 
to a ray of the line $a$, -- a contradiction to the condition that the pair $(a,b)$ is transformed to $(d,c)$ by rotation.
\end{proof}

{\bf Claim 2} {\it There exist no 4-reflective orbits with distinct reflection laws at some pair of opposite vertices.}

\begin{proof} Suppose the contrary: there exists a 4-reflective orbit $ABCD$, say, with usual reflection law at $A$ and skew at $C$. 
Then the rays $OB$ and $OD$ are separated by the line $c$ and are not separated by the line $a$, -- a contradiction to Claim 1. 
This proves Claim 2.
\end{proof}

{\bf Claim 3.} {\it There exists no 4-reflective orbit with usual reflection law at each vertex.}

\begin{proof} The reflection law  at the vertex with maximal distance to $O$ is obviously skew.
\end{proof}

Let $\alpha$ ($\beta$) denote the smallest angle between the lines $a$ and $b$ (respectively, $c$ and $b$). Then the smallest 
angle between the lines $d$ and $c$ ($a$ and $d$) equals $\alpha$ (respectively, $\beta$), by isometry of pairs $(a,b)$ and $(d,c)$. 
One has $\alpha\neq\beta$, since the lines are distinct. 
Let us name the mirrors so that $\alpha<\beta$: then the line $b$ is closer than $d$ to the line $a$. One can achieve this by 
interchanging opposite mirrors $a$ and $c$. 

{\bf Claim 4.} {\it There exists a 4-reflective orbit $ABCD$ with usual reflection law at $A$ (and hence, $C$).}

\begin{proof} Fix a $\gamma\in(\alpha,\beta)$, an $A\in a\setminus O$ and a $B\in b\setminus O$ such that the angle between the rays 
$OA$ and $AB$ equals $\gamma$. Then the line $AD$ symmetric to $AB$ with respect to the line $a$ intersects $d$ at a point $D$ 
lying together with $B$ in the same half-plane with respect to the line $a$, by the inequality 
$\angle(a,d)=\beta>\angle(a,AD)=\angle(a,AB)=\gamma$. 
The triangle $BAD$ thus constructed obviously extends to a 4-reflective orbit $ABCD$, since 
the composition of symmetries with respect to the mirrors is identity. By construction, the reflection law at $A$ is usual. The claim is proved. 
\end{proof}

{\bf Claim 5.} {\it There exists a 4-reflective orbit with skew reflection law at each vertex.}

\begin{proof} Let us take the above-constructed 
4-reflective orbit $ABCD$ with usual reflection law at $A$. Let $b(A)$ denote the line through $A$ 
parallel to $b$, let $B_{\infty}$ be their intersection point  with the infinity line. 
Let us degenerate $ABCD$ so that $A$ remains 
fixed and $B$ tends to $B_{\infty}$: the angle inequality $\alpha<\gamma<\beta$ remains valid. 
Then the vertex $C$ tends to the finite intersection 
point $C'$ of the line $c$ and a finite line $b'$ parallel to $b$: the line $b'$ contains the point symmetric to $A$ with respect to the line $b$. 
Let $b^*(A)$ denote the line symmetric to $b(A)$ with respect to $a$. 
The vertex $D$ tends to the intersection point $D'$ of the lines $d$ and $b^*(A)$.  
The limit $D'$ is finite, since the latter lines are not parallel, by the inequality $\alpha<\beta$. By construction, the vertices $A$ and $C'$ 
are separated by the line $b$, since they lie on lines parallel to $b$ and symmetric with respect to it. Therefore, as $B$ crosses 
the infinity line, the reflection law at $B$ remains skew, while the reflection law at $A$ changes from usual to skew. 
Finally, after crossing the infinity line by the vertex $B$ we obtain a new quadrilateral $ABCD$ with skew reflection law at $A$ and $B$. 
The reflection law at the other vertices $C$ and $D$ is also skew, by Claim 2. This proves Claim 5. 
\end{proof}
 
There exists a 4-reflective orbit with usual reflection law at $b$ and $d$, by Claim 4, which also applies to $B$ and $D$, by symmetry. 
This together with Claims 2--5 proves Addendum 2.

{\bf Case of type 3):} 
a pseudo-billiard $a$, $b$, $a$, $b$ with $a$ and $b$ being real confocal conics. The 4-reflectivity of such a pseudo-billiard 
in the case of ellipses is given by  Urquhart' Theorem \ref{tell}. Its proof given in \cite[p.59, corollary 4.6]{tab} applies to the other types of 
confocal conics as well, as was mentioned in loc. cit. Thus, all the pseudo-billiards listed in type 3) are realized. Let us prove the 
reflection law statement of Theorem \ref{class} and Addendum 3. To do this, we consider the 4-reflective orbit set 
$\uor\subset a\times b\times a\times b$ and its closure $\ur=\overline{\uor}$ in the usual topology. 

\begin{remark} The set $\ur$ is an algebraic surface 
with birational projection to $a\times b$ (to positions of any two neighbor vertices), by Proposition \ref{comp-set} and Lemma \ref{lbirat}. 
In the case, when $a$ and $b$ are not parabolas, the latter birational projection is a diffeomorphism, and the 
surface $\ur$ is a torus: the only indeterminacies of the complexified birational projection correspond to isotropic tangencies 
(by Proposition \ref{aaaa}), which 
do not lie in the real domain. 
\end{remark}

We use the next  proposition  
describing the complement $\ur\setminus \uor$. Then we analyze how the reflection laws change as a 4-reflective orbit
 crosses a component of the latter complement. The first step  of this analysis is given by Proposition \ref{crossma} below. 
%

\begin{proposition} \label{class-deg} The complement $\ur\setminus\uor$ is the union of the following sets:

Case (i): $a$ and $b$ are ellipses, $a$ is smaller. The union $\mct_{a,\rr}$ of two disjoint closed curves 
$\mct_{a,\rr}^{\pm}\subset\ur$ consisting of quadrilaterals of type 
$ABCB$ such that $AB$ and $CB$ are tangent to $a$ at $A$ and $C$ respectively. 

Case (ii): $a$ and $b$ are hyperbolas, and $b$ separates the branches of the hyperbola $a$. The union $\mct_{a,\rr}$ of 
two disjoint closed curves $\mct_{a,\rr}^{\pm}$ 
defined as above. Four closed curves $\Gamma_g$, $g=a,b,c,d$; each curve $\Gamma_g$  
consists of quadrilaterals $ABCD$ with infinite vertex $G$. 

Case (iii): $a$ and $b$ are respectively ellipse and hyperbola. The union $\mathcal T_{a,\mathbb R}$ of two disjoint closed 
curves $\mct_{a,\rr}^{\pm}$ and the union $\mathcal T_{b,\mathbb R}$ of two disjoint closed 
curves $\mct_{b,\rr}^{\pm}$ defined as 
in Case (i). Here and in the next cases 
the curves $\mathcal T_{a,\mathbb R}^\pm$ and $\mathcal T_{b,\mathbb R}^\pm$ contain appropriate single-point quadrilaterals 
$AAAA$, $A\in a\cap b$. Two closed curves $\Gamma_g$, $g=b,d$, defined as in Case (ii). 

Case (iv): $a$ and $b$ are codirected parabolas, $a$ lies inside the domain bounded by $b$. Two curves $\mct_{a,\rr}^{\pm}$ as in Case (i) 
that intersect at the point $OOOO$, $O=a\cap b$ is the infinite intersection point. Four closed curves $\Gamma_{ab}^\rr$, 
$\Gamma_{bc}^\rr$, $\Gamma_{cd}^\rr$, $\Gamma_{da}^\rr$ formed by quadrilaterals $OOAB$, $AOOB$, $ABOO$, $OBAO$ 
respectively with variable $A$ and $B$. 

Case (v): $a$ and $b$ are oppositely directed parabolas. The union $\mathcal T_{a,\mathbb R}$ of two closed curves $\mct_{a,\rr}^{\pm}$ 
and the union $\mathcal T_{b,\mathbb R}$ 
of two closed curves $\mct_{b,\rr}^{\pm}$ defined as in Case (i). Four closed curves $\Gamma_{gh}^\rr$ defined as above. 
The curves $\mct_{a,\rr}^{\pm}$ intersect only at the point $OOOO$, as do the curves $\mct_{b,\rr}^{\pm}$. 
\end{proposition}

\begin{proof} The statements of the proposition in Cases (i) and (ii) follow immediately from Theorem \ref{irrdeg} and convexity. 
Let us 
prove them in Case (iii). The complement $\ur\setminus\uor$ consists of 
the analytic set $\mct_{a,\rr}$ of quadrilaterals of type $ABCB$ with tangencies to $a$ 
at $A$ and $C$, analogous analytic set $\mct_{b,\rr}$ and the above curves $\Gamma_g$, by Theorem \ref{irrdeg}. The set $\mct_{a,\rr}$  
consists of two disjoint closed components $\mct_{a,\rr}^{\pm}$. This follows from the same statement on the 
set of pairs $(A,B)\in a\times b$ such that the line $AB$ is tangent to $a$ at $A$. Indeed,  
each intersection point $B$ of a tangent line to the ellipse $a$ with the hyperbola $b$ lies in one of the two arcs forming the 
complement of $b$ to the interior domain of the ellipse $a$. Choice of arc defines a component of the set $\mct_{a,\rr}$. Each component is 
diffeomorphically parametrized by the 
tangency vertex $A\in a$ and hence, is closed. The proof of the same statements for the set $\mathcal T_{b,\mathbb R}$ is analogous. 
The proof in Cases (iv) and (v) is analogous; the curves $\Gamma_{gh}^{\mathbb R}$ are the real parts of the 
curves $\Gamma_{gh}$ from Theorem 5.1 corresponding to the infinity line. 
The proposition is proved.
\end{proof}
 
\begin{proposition} \label{crossma} Every quadrilateral $x=ABCB\in\mct_{a,\rr}^{\pm}$ with finite vertices has a neighborhood 
$V\subset U_{\mathbb R}$ 
such that for every quadrilateral $A'B'C'D'\in V\cap\uor$ the reflection laws at $A'$ and $C'$ are skew, and those at $B'$ and $D'$ are 
either both usual, or both skew, dependently only on $x$.
\end{proposition} 

The proposition follows from definition.

 Let us prove statements of Addendum 3 for each case in Proposition \ref{class-deg} 

{\bf Subcase (i) of confocal ellipses:} Each 4-reflective orbit deforms in $\uor$ to the set 
$\mct_{a,\rr}$. This together with 
Proposition \ref{crossma} and convexity implies that the reflection law is usual at $b$ and skew at $a$ and proves Addendum 3.
 
 {\bf Subcase of confocal hyperbolas:} say, the hyperbola $b=d$ is inside the concave 
 domain between the two branches of the hyperbola $a=c$. The reflection laws for a quadrilateral $x\in\uor$ do not change, as 
 $x$ crosses one of the curves $\mct_{a,\rr}^{\pm}$, $\mct_{b,\rr}^{\pm}$ outside the curves $\Gamma_g$, by Proposition \ref{crossma}. 
 They may change, as $x$ crosses a curve $\Gamma_g$, i.e., some vertex crosses the infinity line. 
 Note that the union of intersection points of different curves 
 $\Gamma_g$ is discrete and does not separate domains. Below we classify reflection law combinations near arbitrary 
 quadrilateral $A_0B_0C_0D_0\in\ur$ with unique infinite vertex: either $A_0$, or $B_0$ (cases of $C_0$ or $D_0$ are analogous). 
 This together with the previous statement and connectivity of the surface $\ur$ 
 implies that this classification covers all the possible reflection law combinations. 
%
%
 
 a) $x=A_0B_0C_0D_0\in \Gamma_a$: only the vertex $A_0$ is infinite. 
 Set $L_a=T_aA_0$. There exists a 
 path $\gamma(t)=A_0B_tC_tD_t\subset \Gamma_a\setminus(\Gamma_b\cup\Gamma_d)$, $t\in[0,1]$, $\gamma(0)=x$ 
such that $B_1,D_1\in L_a\cap b$: then $\gamma(1)\in\mathcal T_{a,\mathbb R}$ and $B_1=D_1$. The deformed vertices 
$B_t$ and $D_t$ do not cross the infinity line, hence they lie on the same affine branch $b_0$ of the hyperbola $b$, as for $t=1$. Let $a_0$ 
denote the affine branch of the hyperbola $a$ that is contained in the convex domain bounded by $b_0$. For every $t$ one has $C_t\in a_0$, 
since the lines $A_0B_t$ and $A_0D_t$ are parallel and by focusing property of the convex mirror $b_0$. Hence, the path $\gamma$ is 
disjoint from the union $\Gamma_b\cup\Gamma_c\cup\Gamma_d$.   Fix a small neighborhood $W\subset U_{\mathbb R}\setminus
(\Gamma_b\cup\Gamma_c\cup\Gamma_d)$ of the point $\gamma(1)$. Every quadrilateral in $U_{\mathbb R}^0$ close to $x$ can be 
deformed to a quadrilateral in $W$ along a path in $U_{\mathbb R}^0$ close to $\gamma$; the reflection laws in the deformed quadrilaterals 
remain constant. Thus, it suffices to describe the types of quadrilaterals in $W\cap U_{\mathbb R}^0$. For every quadrilateral 
$ABCD\in W\cap U_{\mathbb R}^0$ the reflection laws at $B$ and $D$ coincide (easily follows from Proposition 6.7). As a quadrilateral in 
$W$ crosses $\Gamma_a$ (i.e., as $A$ crosses the infinity line), the reflection law at $B$ and $D$ changes to opposite. This implies that 
every quadrilateral in $W\cap U_{\mathbb R}^0$, and hence, every quadrilateral in $U_{\mathbb R}^0$ close to $x$ has one of the two  
types depicted at Fig. 8 a), b).

 b) $x=A_0B_0C_0D_0\in\Gamma_b$: only the vertex $B_0$ is infinite. 
 There exists a path $\gamma(t)=A_tB_0C_tD_t$, $t\in[0,1]$, $\gamma(0)=x$, $\gamma(1)=A_1B_0C_1D_1
\in\mathcal T_{a,\mathbb R}$ (thus, $D_1=B_0$), $\gamma[0,1)\subset\Gamma_b\setminus(\Gamma_a\cup\Gamma_c\cup\Gamma_d)$, as 
in the above discussion. As above, it suffices to describe the types of quadrilaterals in $U_{\mathbb R}^0$ close to $\gamma(1)$. The points 
$A_1$ and $C_1$ lie on different affine branches of the hyperbola $a$:  the lines $B_0A_1$ and $B_0C_1$ are parallel to the asymptotic 
line $T_{B_0}b$ and are symmetric with respect to it, and $T_{B_0}b$ separates the branches of the hyperbola $a$. In every 
quadrilateral $ABCD\in U_{\mathbb R}^0$ close to $\gamma(1)$ the vertices $B$ and $D$ are either on the same branch of the hyperbola $b$, 
 or  on different branches. Then we get   Fig.8b) or Fig.8c) respectively (up to interchanging names of either 
 $B$ and $D$, or $A$ and $C$, or both). This proves Addendum 3.

{\bf Subcase of confocal ellipse and hyperbola.} Let $a=c$ be an ellipse, and $b=d$ be a confocal hyperbola. It suffices to describe 
the types of quadrilaterals in $U_{\mathbb R}^0$ close to  a quadrilateral $x=A_0B_0C_0D_0$ with exactly one 
infinite point, say $B_0$, as in the previous case. The above deformation argument reduces this problem to the description of types of 
quadrilaterals in $U_{\mathbb R}^0$ close to a quadrilateral $y=A_1B_0C_1B_0\in\mathcal T_{a,\mathbb R}$. In every quadrilateral 
$ABCD\in U_{\mathbb R}^0$ close to $y$ the points $B$ and $D$ lie either on the same branch of the hyperbola $b$, or on different branches. 
We get Figures 9a) and 9b) respectively. This proves Addendum 3. 

{\bf Subcase of confocal parabolas.} Any two confocal parabolas are either codirected (Fig.10a)), or oppositely directed (Fig.10b)). 
Let us consider the first case of codirected parabolas, say, $a$ is contained inside the convex domain bounded by $b$. Fix a quadrilateral 
$x=A_0B_0C_0D_0\in\uor$. Its  reflection law at $b$ is usual, by convexity. We claim that the reflection law at $a$ is skew, i.e., 
$x$ is as at   Fig.10a). To do this, it suffices to show that $x$ deforms in $\uor$ 
to a quadrilateral $y\in\mct_{a,\rr}$  (Proposition \ref{crossma}). 
There exists a continuous deformation $A_0B_tC_tD_t$, $t\in[0,1]$ with finite $B_t$ and $D_t$ for all $t$ to 
a quadrilateral $y=A_0B_1C_1D_1$ with $A_0B_1=T_{A_0}a$, since 
the exterior parabola $b$ has only one infinite point. Then $y\in\mct_{a,\rr}$, hence $B_1=D_1$.  The 
vertices $C_t$ also remain finite:  each quadrilateral in $\ur$ with 
an infinite vertex has at least two infinite vertices, being contained   in one of the curves $\Gamma_{gh}^\rr$ from Propisiton \ref{class-deg}. 
This together with the above discussion implies that the quadrilateral $x$ is as at Fig.10a). 

%

Let us now consider the case of oppositely directed parabolas. It suffices to describe reflection laws only in those quadrilaterals in $\uor$ 
that are close to the union of curves $\Gamma_{gh}^\rr$, as in the case of hyperbolas. That is, fix a 
quadrilateral $x=A_0B_0C_0D_0\in\uor$, say, with $A_0$ and $B_0$ 
close to infinity:  $A_0$ ($B_0$)  lies outside the convex domain bounded by the parabola $b$ (respectively, $a$). 
The quadrilateral $x$ deforms  to $\mct_{a,\rr}$ in $\uor$, as in the above subcase. 
This implies that the reflection laws at $A_0$ and $C_0$ are skew, as in the same subcase. Similarly, $x$ deformes 
to $\mct_{b,\rr}$, hence the reflection laws at $B_0$ and $D_0$ are also skew. Thus, the quadrilateral $x$ is 
as at Fig.10b). The proof of Theorem \ref{tclass} and its addendums is complete. 
\end{proof}

\begin{example} The pseudo-billiard $a$, $b$, $a$, $b$ formed by two orthogonal lines $a$ and $b$ is 4-reflective of type 1), and 
its 4-reflective orbits are symmetric rhombi, see Fig.10a). It is a limit of a 4-reflective pseudo-billiard 
of type 3) on confocal ellipse and hyperbola, as the ellipse tends to a segment and the hyperbola tends to the orthogonal line through 
the center of the limit segment. It can be also viewed as a limit of a 4-reflective pseudo-billiard on a pair of 
confocal hyperbolas (parabolas) degenerating to orthogonal lines. 
\end{example}

\begin{figure}[ht]
  \begin{center}
   \epsfig{file=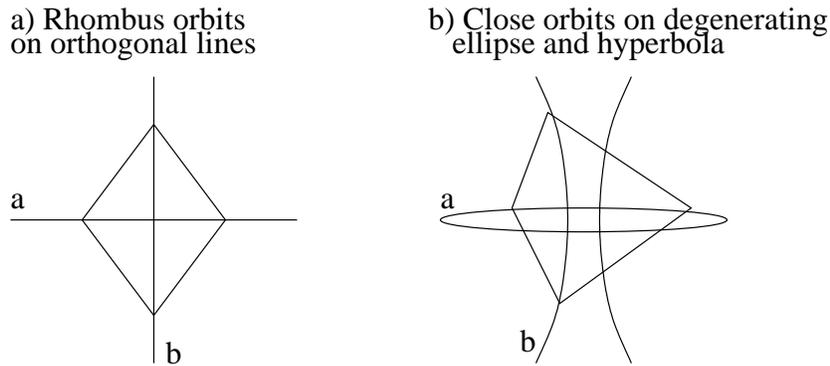}
    \caption{Orthogonal line billiard as a limit  of degenerating  confocal ellipse-hyperbola billiard}
  \end{center}
\end{figure}

\section{Acknowledgements}

I wish to thank Yu.S.Ilyashenko, Yu.G.Kudryashov and A.Yu.Plakhov, who had attracted my attention respectively to 
Ivrii's conjecture and its analogue in the invisibility theory. I am grateful to them and to S.V.Bolotin, E.Ghys,  A.V.Klimenko, 
V.S.Kulikov, K.A.Shramov, B.Sevennec, 
J.-C.Sikorav, S.L.Tabachnikov,  V.A.Timorin, R.Uribe-Vargas  for helpful discussions.  \  I \ wish \  to thank 
A.L.Gorodentsev and V.Zharnitsky for helpful  discussions and very valuable remarks on the text.

\end{document}